\documentclass[]{interact}
\usepackage{epstopdf}
\usepackage[caption=false]{subfig}

\usepackage[numbers,sort&compress]{natbib}
\bibpunct[, ]{[}{]}{,}{n}{,}{,}

\theoremstyle{plain}
\newtheorem{theorem}{Theorem}[section]

\newtheorem{proposition}[theorem]{Proposition}
\theoremstyle{definition}

\theoremstyle{remark}
\newtheorem{remark}{Remark}

\usepackage{listings}
\usepackage{pdflscape}
\usepackage{amsfonts}
\usepackage{epsfig}
\usepackage{lscape}
\usepackage{longtable}
\usepackage{amsmath,amssymb,amsthm}
\usepackage{graphicx}
\usepackage{color}
\usepackage{placeins}
\usepackage{url}
\usepackage{cases}
\usepackage{longtable}
\usepackage{supertabular}
\usepackage{multirow}
\newcommand{\tabincell}[2]{\begin{tabular}{@{}#1@{}}#2\end{tabular}}
\usepackage{amsmath}
\allowdisplaybreaks[4]
\usepackage{cite}
\usepackage{algorithmic}
\usepackage{algorithm,float}
\usepackage{booktabs}
\usepackage{enumitem}
\usepackage{array}
\usepackage{lipsum}



\begin{document}

\articletype{}

\title{A Highly Efficient Algorithm for Solving Exclusive Lasso Problems \thanks{In memory of Oleg Burdakov.}\thanks{The first two authors contribute equally.}}

\author{
\name{Meixia Lin\textsuperscript{a}, Yancheng Yuan\textsuperscript{b}, Defeng Sun\textsuperscript{b} \thanks{CONTACT Defeng Sun. Email: {\tt defeng.sun@polyu.edu.hk}.}, and Kim-Chuan Toh\textsuperscript{c}}
\affil{\textsuperscript{a}Engineering Systems and Design, Singapore University of Technology and Design, Singapore; \textsuperscript{b}Department of Applied Mathematics, The Hong Kong Polytechnic University, Hung Hom, Hong Kong;  \textsuperscript{c}Department of Mathematics and Institute of Operations Research and Analytics, National University of Singapore, Singapore}
}

\maketitle

\begin{abstract}
The exclusive lasso (also known as elitist lasso) regularizer has become popular recently due to its superior performance on intra-group feature selection. Its complex nature poses difficulties for the computation of high-dimensional machine learning models involving such a regularizer. In this paper, we propose a highly efficient dual Newton method based proximal point algorithm (PPDNA) for solving large-scale exclusive lasso models. As important ingredients, we systematically study the proximal mapping of the weighted exclusive lasso regularizer and the corresponding generalized Jacobian. These results also make popular first-order algorithms for solving exclusive lasso models more practical. Extensive numerical results are presented to demonstrate the superior performance of the PPDNA against other popular numerical algorithms for solving the exclusive lasso problems.
\end{abstract}


\begin{amscode}
90C06; 90C25; 90C90
\end{amscode}

\section{Introduction}
\label{sec:intro}

For a given feature matrix $A=[a_1,a_2,\cdots,a_n] \in \mathbb{R}^{m \times n}$, we are interested in the machine learning models of the form:
\begin{align}
	\min_{x \in \mathbb{R}^n} \ \Big\{f(x):=  h(Ax)-\langle c,x\rangle + \lambda p(x)\Big\},\label{eq: general_ml_model}
\end{align}
where $c \in \mathbb{R}^{n}$ is a given vector, $h: \mathbb{R}^m \to \mathbb{R}$ is a convex twice continuously differentiable function, $p: \mathbb{R}^n \to (-\infty,+\infty]$, a closed and proper convex function, is a regularizer which usually enforces feature selection to prevent overfitting, and $\lambda > 0$ is a hyper-parameter which controls the trade-off between the loss function and the regularizer.

Many regularizers have been proposed to enforce sparsity with desirable structure in the predictors learned by machine learning models. For example, the lasso model \citep{tibshirani1996regression} can induce sparsity in the predictors but without structured patterns, and the group lasso model \citep{yuan2006model} can induce inter-group level sparsity. In some applications, intra-group level sparsity is desirable, which means that not only features from different groups, but also features in a seemingly cohesive group are competing to survive. One application comes from performing portfolio selections both across and within sectors in order to diversify the risk across different sectors. To achieve this intra-group sparsity, the exclusive lasso regularizer was proposed in \citep{zhou2010exclusive} (also named as elitist lasso \citep{kowalski2009sparse}), originally for multi-task learning. Since then, it has also been widely used in other applications  such as image processing \citep{zhang2016robust}, sparse feature clustering \citep{yamada2017localized} and nuclear magnetic resonance (NMR) spectroscopy \citep{campbell2017within}. Let $w\in \mathbb{R}_{++}^n$ be a weight vector and $\mathcal{G} := \{g_1,\cdots,g_l\}$ be an index partition of the features such that $\bigcup_{j=1}^l g_j = \{1, 2, \dots, n\}$ and $g_j \bigcap g_k = \emptyset$ for any $j\neq k$. The corresponding  weighted exclusive lasso regularizer is defined as
\begin{align}
	\Delta^{\mathcal{G},w}(x) :=\sum_{j=1}^l \|w_{g_j} \circ x_{g_j}\|_1^2, \quad \forall\, x \in \mathbb{R}^n,\label{eq: exclusive-lasso-regularizer}
\end{align}
where `$\circ$' denotes the Hadamard product, and $x_{g_j}$ denotes the sub-vector extracted from $x$ based on the index set $g_j$. Naturally, when solving exclusive lasso models, we can expect that each $x_{g_j}$ has nonzero coordinates under mild conditions, which means that every group has representatives.

Existing algorithms for solving exclusive lasso models, such as the iterative least squares algorithm (ILSA) \citep{kong2014exclusive,yamada2017localized}, the coordinate descent (CD) method \citep{campbell2017within}, are very time-consuming to obtain a solution with moderate accuracy, even for problems with medium sizes. In addition, popular first-order algorithms, such as the accelerated proximal gradient method (APG) \citep{zhang2016robust}, FISTA \citep{beck2009fast} and the alternating direction method of multipliers (ADMM) \citep{eckstein1992douglas,glowinski1975approximation}, have not been widely used to solve exclusive lasso models. The main reason may lie in the fact that the proximal mapping of the exclusive lasso regularizer, which is the key ingredient for the efficient implementation of the algorithms mentioned above, has not been systematically studied yet. Yoon and Hwang provided a procedure for computing the proximal mapping of $\Delta^{\mathcal{G},w}(\cdot)$ in \citep{yoon2017combined} but unfortunately it is mathematically incorrect. Kowalski mentioned the proximal mapping in \citep{kowalski2009sparse}, but the derivation contains some errors, and this result is not known to most researchers in the optimization and machine learning communities. In this paper, we systematically study the exclusive lasso regularizer, and provide an $O(n\log n )$ routine to compute the proximal mapping of the general weighted exclusive lasso regularizer \eqref{eq: exclusive-lasso-regularizer}. Such an $O(n\log n)$ procedure is important for the practical efficiency of many algorithmic frameworks, such as APG and ADMM, for solving exclusive lasso models. However, as we shall see in the numerical experiments, even with the $O(n\log n)$ procedure to compute the proximal mapping, first-order algorithms, such as APG and ADMM, are not efficient enough. To overcome this computational challenge, we design a highly efficient second-order type algorithm, the dual Newton method based proximal point algorithm (PPDNA), to solve exclusive lasso models. As a key ingredient of the PPDNA, we carefully derive the generalized Jacobian of the proximal mapping of the weighted exclusive lasso regularizer. We also analyse the underlying structures of the generalized Jacobian to facilitate its efficient computation within the semismooth Newton method. Numerical results demonstrate the superior performance of the PPDNA against ADMM, APG, CD, and ILSA for solving exclusive lasso models.

We summarize our main contributions in this paper as follows.
\begin{itemize}[noitemsep,topsep=0pt]
	\item[1.] We develop a highly efficient dual Newton method based proximal point algorithm to solve the exclusive lasso model. We prove that the error bound condition holds for commonly used exclusive lasso models, which guarantees the superlinear convergence of the preconditioned proximal point algorithm.
	\item[2.] As key ingredients of the PPDNA, we systematically study the proximal mapping of the weighted exclusive lasso regularizer, and the corresponding generalized Jacobian. These results are also critical in computing the key projection step of various first-order algorithms for solving exclusive lasso models.
	\item[3.] We demonstrate numerically that the PPDNA is highly efficient and robust when comparing to the state-of-the-art algorithms for solving exclusive lasso models. Furthermore, we apply the exclusive lasso model in some real application problems including the index exchange-traded fund, and image and text classifications.
\end{itemize}

The rest of the paper is organized as follows. In Section \ref{sec:pppa}, we design a preconditioned proximal point algorithm for solving exclusive lasso models. As important ingredients, in Section \ref{sec:proxJacobian}, we systematically study the weighted exclusive lasso regularizer, through providing an $O(n\log n)$ procedure to compute the proximal mapping and its generalized Jacobian. Based on these results, we develop a dual Newton method for solving the subproblems of the preconditioned proximal point algorithm in Section \ref{sec:ssn}. Numerical experiments on various synthetic data are presented in Section \ref{sec:numerical}, which demonstrate the superior performance of the PPDNA against the state-of-the-art algorithms for solving the exclusive lasso models. More interesting experiments on real applications of the exclusive lasso model are also presented. In the end, we conclude the paper in Section \ref{sec:conclusion}.

\vspace{0.2cm}
\noindent\textbf{Notation and preliminaries:} We use $\mathbb{R}^n$ to denote the space of  $n$-dimensional vectors and  $\mathbb{R}^n_{+}$ (resp., $\mathbb{R}^n_{++}$) to denote the space of vectors in $\mathbb{R}^n$ with nonnegative (resp., positive) elements. We let $\mathbb{S}^n$ be the space of all $n\times n$ real symmetric matrices. For any $z\in \mathbb{R}$, ${\rm sign}(z)$ denotes the sign function of $z$, that is ${\rm sign}(z) = 1$ if $z>0$; ${\rm sign}(z) = 0$ if $z=0$; ${\rm sign}(z) = -1$ if $z<0$. Define $z^{+}:=\max\{z,0\}$, $z^{-}:=\min\{z,0\}$. We use `${\rm Diag}(x)$' to denote the diagonal matrix whose diagonal is given by the vector $x$, and use `${\rm Diag}(X_1,\cdots,X_n)$' to denote the block diagonal matrix whose $i$-th diagonal block is the matrix $X_i$, $i=1,\cdots,n$.  Let ${\cal M}:\mathbb{R}^n\rightarrow \mathbb{R}^n$ be any self-adjoint positive semidefinite linear operator. We define $\langle x,x'\rangle_{\cal M}:=\langle x,{\cal M}x'\rangle$, and $\|x\|_{\cal M}:=\sqrt{\langle x,x\rangle_{\cal M}}$ for all $x,x'\in \mathbb{R}^n$. For a given subset ${\cal C}$ of $\mathbb{R}^n$, we denote the weighted distance of $x\in \mathbb{R}^n$ to ${\cal C}$ as ${\rm dist}_{\cal M}(x,{\cal C}):=\inf_{x'\in {\cal C}}\|x-x'\|_{\cal M}$. The largest eigenvalue of ${\cal M}$ is denoted as $\lambda_{\max}({\cal M})$.

Let  $q:\mathbb{R}^n\rightarrow (-\infty,\infty]$ be a  closed and proper convex function. The conjugate function of $q$ is defined as $q^*(z):=\sup_{x\in\mathbb{R}^n}\{\langle x,z\rangle-q(x)\}$. The Moreau envelope of $q$ at $x$ is defined by
\begin{align*}
	{\rm E}_q(x):=\min_{y\in \mathbb{R}^n}\Big\{ q(y)+\frac{1}{2}\|y-x\|^2\Big\},
\end{align*}
and the corresponding proximal mapping ${\rm Prox}_q(x)$ is defined as the unique optimal solution of the above problem. It is known that for any $x\in \mathbb{R}^n$, $\nabla {\rm E}_q(x)=x-{\rm Prox}_q(x)$, and ${\rm Prox}_q(\cdot)$ is Lipschitz continuous with modulus $1$ \citep{moreau1965proximite,rockafellar1976monotone}.

In order to study the weighted exclusive lasso regularizer $\Delta^{\mathcal{G},w}(\cdot)$ defined in \eqref{eq: exclusive-lasso-regularizer}, we use the following notations. For $j=1\cdots,l$, we define the linear mapping ${\cal P}_j:\mathbb{R}^n \rightarrow \mathbb{R}^{|g_j|}$ as ${\cal P}_j x=x_{g_j}$ for all $x\in \mathbb{R}^n$, and ${\cal P}=[{\cal P}_1;\cdots;{\cal P}_l]$. Let $n_j=\sum_{k=1}^j |g_k|$ and $n_0=0$. Denote $x^{(j)}$ as the sub-vector extracted from $x$ based on the index set $\{n_{j-1}+1,n_{j-1}+2,\cdots,n_{j}\}$ for $j=1,\cdots,l$. According to these notations, we have
\begin{align}
	\Delta^{\mathcal{G},w}(x) =\sum_{j=1}^l\|({\cal P}w)^{(j)}\circ ({\cal P}x)^{(j)}\|_1^2, \quad \forall\, x \in \mathbb{R}^n.\label{eq: reformulation_exclusive}
\end{align}

\section{A preconditioned proximal point algorithm for exclusive lasso models}
\label{sec:pppa}
We focus on the machine learning model \eqref{eq: general_ml_model} with the weighted exclusive lasso regularizer defined in \eqref{eq: exclusive-lasso-regularizer}, which is also called the weighted exclusive lasso model. Denote the optimal solution set of the problem \eqref{eq: general_ml_model} as $\Omega$. Throughout this paper, we assume that the solution set $\Omega$ is nonempty and compact. For many popular machine learning models involving the exclusive lasso regularizers, this assumption is satisfied automatically, as discussed in \citep[Section 2.1]{zhou2017unified}.

As we have mentioned in Section \ref{sec:intro}, existing algorithms \citep{campbell2017within,kong2014exclusive,yamada2017localized} face difficulties even for solving medium-scale exclusive lasso models to a moderate accuracy. To overcome this challenge, in this paper, we aim to design a highly efficient preconditioned proximal point algorithm (PPA) to solve the convex composite programming problem \eqref{eq: general_ml_model}. We further prove that for exclusive lasso models, a certain error bound condition holds, which guarantees that the preconditioned PPA for solving the weighted exclusive lasso models has an asymptotic superlinear convergence rate.

\subsection{A preconditioned PPA algorithmic framework}
For any starting point $x^0\in \mathbb{R}^n$, the preconditioned PPA generates a sequence $\{x^k\}\subseteq \mathbb{R}^n$ by the following approximate rule for solving \eqref{eq: general_ml_model}:
\begin{align}
	x^{k+1}\approx {\cal P}_k(x^k)
	:=\underset{x\in\mathbb{R}^n}{\arg\min}\ \left\{
	\begin{aligned}
    f_k(x):=  & \ h(Ax)  - \langle c,x\rangle +\lambda  p(x)\\
	&+\frac{1}{2\sigma_k}\|x-x^k\|^2+\frac{\tau}{2\sigma_k}\|Ax-Ax^k\|^2
	\end{aligned}
	\right\},\label{eq:pre_ppa}
\end{align}
where $\{\sigma_k\}$ is a sequence of nondecreasing positive real numbers $(\sigma_k\uparrow \sigma_{\infty}\leq \infty)$, $\tau> 0$ is a given parameter.

Comparing to the classical proximal point algorithm, the addition of the second proximal term $\frac{\tau}{2\sigma_k}\|Ax-Ax^k\|^2$ is critical for us to obtain the dual of \eqref{eq:pre_ppa} as a smooth unconstrained problem. We equivalently rewrite the minimization problem \eqref{eq:pre_ppa} as a constrained optimization problem:
\begin{align}
	\min_{x\in \mathbb{R}^n, y\in \mathbb{R}^m} \ \Big\{
	h(y)  - \langle c,x\rangle +\lambda  p(x) +\frac{1}{2\sigma_k}\|x-x^k\|^2+\frac{\tau}{2\sigma_k}\|y-Ax^k\|^2 \mid Ax - y = 0
	\Big\}.\label{eq:pre_ppa_constrained}
\end{align}	
By introducing the Lagrangian multiplier $u\in\mathbb{R}^m$, the associated Lagrangian function is
\begin{align*}
	l(x,y;u)&:= h(y)  - \langle c,x\rangle +\lambda  p(x) +\frac{1}{2\sigma_k}\|x-x^k\|^2+\frac{\tau}{2\sigma_k}\|y-Ax^k\|^2 + \langle u, Ax - y\rangle\\
	&= h(y) + \frac{\tau}{2\sigma_k}\|y-Ax^k -\frac{\sigma_k}{\tau}u\|^2 + \frac{\tau}{2\sigma_k}\|Ax^k\|^2 - \frac{\tau}{2\sigma_k}\| Ax^k +\frac{\sigma_k}{\tau}u\|^2\\
	& \quad +\,\lambda p(x) + \frac{1}{2\sigma_k} \|x-x^k-\sigma_k c +\sigma_k A^T u\|^2 +\frac{1}{\sigma_k}\|x^k\|^2 - \frac{1}{2\sigma_k} \|x^k+\sigma_k c -\sigma_k A^T u\|^2.
\end{align*}
Therefore, the dual problem of \eqref{eq:pre_ppa}, i.e., $\max_{u} \min_{x, y} l(x,y;u)$, takes the form of
\begin{align}
	\max_{u\in \mathbb{R}^m}\ \Big\{\psi_k(u)&:= -\frac{\tau}{2\sigma_k}\|Ax^k+\frac{\sigma_k}{\tau}u\|^2 +\frac{\tau}{\sigma_k}{\rm E}_{\sigma_k h/\tau}(Ax^k+\frac{\sigma_k}{\tau}u) +\frac{\tau}{2\sigma_k}\|Ax^k\|^2\notag\\
	&- \frac{1}{2\sigma_k}\|x^k+\sigma_k c-\sigma_k A^T u\|^2 +\frac{1}{\sigma_k}{\rm E}_{\sigma_k\lambda  p}(x^k+\sigma_k c-\sigma_kA^Tu)+\frac{1}{2\sigma_k}\|x^k\|^2\Big\}.\label{eq:ATA_D}
\end{align}
Moreover, the Karush–-Kuhn–-Tucker(KKT) conditions associated with \eqref{eq:pre_ppa_constrained} and \eqref{eq:ATA_D} are
\begin{align}
	\left\{
	\begin{aligned}
		& x={\rm Prox}_{\sigma_k \lambda p}(x^k + \sigma_k c - \sigma_k A^Tu),\\
		& 0 = \nabla h(y) +\tau (y-Ax^k)/\sigma_k -u,\\
		& Ax-y=0.
	\end{aligned}
	\right. \label{eq: ppa_kkt}
\end{align}

Due to the continuous differentiability of the Moreau envelope, the objective function in the optimization problem \eqref{eq:ATA_D} is continuously differentiable. The Lipschitz continuity of the proximal mapping further inspires us to design a highly efficient nonsmooth Newton algorithm to solve the problem \eqref{eq:ATA_D}. The details are discussed in Section \ref{sec:ssn}. As long as we obtain a solution $\bar{u}^{k+1}$ to the problem \eqref{eq:ATA_D}, according to the first equation in the KKT system \eqref{eq: ppa_kkt}, the update of $x$ in the preconditioned PPA iteration \eqref{eq:pre_ppa} can be given as
\begin{align*}
	\bar{x}^{k+1}={\rm Prox}_{\sigma_k \lambda p}(x^k + \sigma_k c - \sigma_k A^T\bar{u}^{k+1}).
\end{align*}

We give the full description of the preconditioned PPA for solving the general machine learning model \eqref{eq: general_ml_model} in Algorithm \ref{alg:ppdna}. Note that the subproblem of the preconditioned PPA is allowed to be solved approximately. To ensure the convergence of the preconditioned PPA, we use the implementable stopping criteria \eqref{eq:stopA_ATA} and \eqref{eq:stopB_ATA} based on the duality gap of \eqref{eq:pre_ppa} and \eqref{eq:ATA_D}.

\begin{algorithm}[H]\small
	\caption{A preconditioned proximal point algorithm for \eqref{eq: general_ml_model}}
	\label{alg:ppdna}
	\begin{algorithmic}[1]
		\STATE \textbf{Input}: $\tau>0$, $0<\sigma_0\leq \sigma_{\infty}\leq \infty$; summable nonnegative sequences $\{\epsilon_k\}_{k=1}^{\infty}$, $\{\delta_k\}_{k=1}^{\infty}$ with $\delta_k<1$, $\forall k$.
		\STATE \textbf{Output}: an approximate optimal solution $x$ to \eqref{eq: general_ml_model}.
		\STATE \textbf{Initialization}: choose $x^0\in \mathbb{R}^n$, $k=0$.
		\REPEAT
		\STATE {\bfseries Step 1}. Find an approximate maximizer $u^{k+1}$ to the problem \eqref{eq:ATA_D} such that the stopping criteria
		\begin{align}
			f_k(x^{k+1})-\psi_k(u^{k+1})&\leq \frac{\epsilon_k^2}{2\sigma_k},\tag{A}\label{eq:stopA_ATA}\\
			f_k(x^{k+1})-\psi_k(u^{k+1})&\leq \frac{\delta_k^2}{2\sigma_k} \|x^{k+1}-x^k\|^2+\frac{\tau\delta_k^2}{2\sigma_k} \|Ax^{k+1}-Ax^k\|^2,\tag{B}\label{eq:stopB_ATA}
		\end{align}
		are satisfied with
		\begin{align*}
			x^{k+1} = {\rm Prox}_{\sigma_k\lambda  p}(x^k+\sigma_k c-\sigma_k A^Tu^{k+1}).
		\end{align*}
		\\[3pt]
		\STATE {\bfseries Step 2}. Update $\sigma_{k+1} \uparrow \sigma_{\infty} \leq \infty$, $k\leftarrow k+1$.
		\UNTIL{Stopping criterion is satisfied.}
	\end{algorithmic}
\end{algorithm}

As for the parameters in the above algorithm, in practice, one can set $\tau = 1/\lambda_{\max}(AA^T)$, $\epsilon_k =\delta_k = 0.5/1.06^k$ and $\sigma_k = 3^{\left \lfloor{k/2}\right \rfloor}$. The following theorem states the global convergence and the asymptotic superlinear convergence rate of the preconditioned PPA for solving \eqref{eq: general_ml_model}.

\begin{theorem}\label{thm:convergence_ALM}
	Let $\{(x^k,u^k)\}$ be the sequence generated by Algorithm \ref{alg:ppdna}.
	\begin{enumerate}
		\item Suppose in {\bfseries Step 1}, the stopping criterion \eqref{eq:stopA_ATA} is satisified at each iteration. Then the sequence $\{x^k\}$ is bounded and $\{x^k\}$ converges to some $x^*\in \Omega$.
		\item Assume that there exists a constant $\kappa>0$ such that ${\cal T}_f:= \partial f$ satisfies the following error bound condition
		\begin{align}
			{\rm dist}(x,\Omega)\leq \kappa {\rm dist}(0,{\cal T}_f(x)),\quad \forall x\in \mathbb{R}^n   \mbox{ satisfying } {\rm dist}(x,\Omega)\leq \sum_{i=0}^{\infty}\epsilon_k+{\rm dist}_{{\cal M}}(x^0,\Omega), \label{error_bound}
		\end{align}
		where ${\cal M}:=I_n+\tau A^T A$ is a positive definite operator on $\mathbb{R}^n$. Suppose in {\bfseries Step 1}, the stopping criteria \eqref{eq:stopA_ATA} and \eqref{eq:stopB_ATA} are both satisified at each iteration. Then it holds for all $k\geq 0$ that
		\begin{align*}
			{\rm dist}_{{\cal M}}(x^{k+1},\Omega)\leq \mu_k {\rm dist}_{{\cal M}}(x^{k},\Omega),
		\end{align*}
		where $\mu_k=(1-\delta_k)^{-1}\Big[\delta_k+ (1+\delta_k)\kappa\zeta (\sigma_k^2+\kappa^2\zeta^2)^{-1/2}\Big]\rightarrow \mu_{\infty}=\kappa\zeta(
				\sigma_{\infty}^2+\kappa^2\zeta^2)^{-1/2}<1$, as $k\rightarrow \infty$, with $\zeta:=1+\tau\lambda_{\max}(A^T A)$.
	\end{enumerate}
\end{theorem}
\begin{proof}
	In order to prove the convergence result, we first need to characterize the stopping criteria \eqref{eq:stopA_ATA} and \eqref{eq:stopB_ATA}. According to the definition of ${\cal P}_k(x^k)$ in \eqref{eq:pre_ppa}, we have that ${\cal P}_k(x^k)=\arg\min f_k(x)$ and $0\in \partial f_k({\cal P}_k(x^k))$. It can be seen from \citep[Exercise 8.8]{rockafellar2009variational} that $\partial f_k(x)=\partial f(x)+(1/\sigma_k) {\cal M}(x-x^k)$, which means there exists $v\in \partial f({\cal P}_k(x^k))$ such that
	\begin{align*}
		0=v+\frac{1}{\sigma_k} {\cal M}({\cal P}_k(x^k)-x^k).
	\end{align*}
	Then it holds that
	\begin{align*}
		&f_k(x^{k+1})-f_k({\cal P}_k(x^k))=f(x^{k+1})-f({\cal P}_k(x^k))+\frac{1}{2\sigma_k}\|x^{k+1}-x^k\|_{{\cal M}}^2-\frac{1}{2\sigma_k}\|{\cal P}_k(x^k)-x^k\|_{{\cal M}}^2\\
		&\geq \langle v,x^{k+1}-{\cal P}_k(x^k)\rangle+\frac{1}{2\sigma_k}\langle x^{k+1}+{\cal P}_k(x^k)-2x_k,x^{k+1}-{\cal P}_k(x^k)\rangle_{{\cal M}}=\frac{1}{2\sigma_k}\|x^{k+1}-{\cal P}_k(x^k)\|_{{\cal M}}^2.
	\end{align*}
	By the strongly duality, we have $f_k({\cal P}_k(x^k))=\inf f_k=\sup \psi_k$. Thus we can see that
	\begin{align*}
		\frac{1}{2\sigma_k}\|x^{k+1}-{\cal P}_k(x^k)\|_{{\cal M}}^2\leq f_k(x^{k+1})-\inf f_k=f_k(x^{k+1})-\sup \psi_k\leq f_k(x^{k+1})-\psi_k(u^{k+1}).
	\end{align*}
	Therefore, we can see that the stopping criterion \eqref{eq:stopA_ATA} implies that
	\begin{align}
		\|x^{k+1}-{\cal P}_k(x^k)\|_{{\cal M}}\leq \sqrt{2\sigma_k\Big(f_k(x^{k+1})-\psi_k(u^{k+1})\Big)\Big)}\leq \sqrt{2\sigma_k \frac{\epsilon_k^2}{2\sigma_k}} =\epsilon_k.\label{eq:stopA_pre_ppa}
	\end{align}
	In addition, the stopping criterion \eqref{eq:stopB_ATA} implies that
	\begin{align}
	\|x^{k+1}-{\cal P}_k(x^k)\|_{{\cal M}} \leq \sqrt{2\sigma_k \frac{\delta_k^2}{2\sigma_k} \|x^{k+1}-x^k\|_{\cal M}^2}=\delta_k\|x^{k+1}-x^k\|_{{\cal M}}.\label{eq:stopB_pre_ppa}
	\end{align}
	Note that the conditions \eqref{eq:stopA_pre_ppa} and \eqref{eq:stopB_pre_ppa} are the general criteria for the approximate calculation in the preconditioned PPA. Then the conclusions of the theorem follow from the fact that $\Omega$ is nonempty together with \citep[Theorem 2.3 \& Theorem 2.5]{li2020asymptotically}.
\end{proof}

\subsection{Error bound conditions for the weighted exclusive lasso models}
As one can see in Theorem \ref{thm:convergence_ALM}, the desired asymptotic superlinear convergence rate of the proposed preconditioned PPA relies on the error bound condition \eqref{error_bound} of ${\cal T}_f$. In this subsection, we establish the error bound condition of ${\cal T}_f$ for the problem \eqref{eq: general_ml_model} with the exclusive lasso regularizer, that is  $p(x)=\Delta^{\mathcal{G},w}(x)=\sum_{j=1}^l \|w_{g_j} \circ x_{g_j}\|_1^2$ in \eqref{eq: general_ml_model}. We are going to prove that the error bound condition of ${\cal T}_f$ holds for the problem \eqref{eq: general_ml_model} with a piecewise linear-quadratic regularizer, which includes the exclusive lasso regularizer as a special case.

For the purpose of analyzing the error bound condition \eqref{error_bound}, we need the proximal residual function $R:\mathbb{R}^n\rightarrow \mathbb{R}^n$ associated with \eqref{eq: general_ml_model}, which is defined as
\begin{align*}
R(x):=x-{\rm Prox}_{\lambda p}(x-A^T \nabla h(Ax)+c),\quad \forall x\in \mathbb{R}^n.
\end{align*}
Indeed, by noting the fact that
\[
\partial f(x) = A^T \nabla h(Ax) -c + \partial (\lambda p)(x),
\]
we know that the first-order optimality condition of \eqref{eq: general_ml_model} is $0\in A^T \nabla h(Ax) -c + \partial (\lambda p)(x)$, which is equivalent to $x = {\rm Prox}_{\lambda p}(x-A^T \nabla h(Ax)+c)$. Therefore, we can see that $\bar{x}\in \Omega$ if and only if $R (\bar{x})=0$. In the following proposition, we prove that the error bound condition with proximal mapping based residual function holds for the problem \eqref{eq: general_ml_model} with a piecewise linear-quadratic regularizer.

\begin{proposition}\label{prop:luotseng}
	For the problem \eqref{eq: general_ml_model}, suppose that $h(\cdot)$ is strongly convex on any compact convex set in $\mathbb{R}^m$ and $p(\cdot)$ is piecewise linear-quadratic. Then  for any $\xi\geq \inf f$, there exist constants $\kappa,\varepsilon>0$ such that
	\begin{align*}
		{\rm dist}(x,\Omega)\leq \kappa \|R(x)\|\mbox{ \quad  for all $x\in \mathbb{R}^n$ with $f(x)\leq \xi$, $\|R(x)\|\leq \varepsilon$}.
	\end{align*}
\end{proposition}
\begin{proof}
	Since $p$ is piecewise linear-quadratic, $p^*$ is also piecewise linear-quadratic by \citep[Theorem 11.14(b)]{rockafellar2009variational}. Thus $\partial p$ and $\partial p^*$ are both polyhedral due to \citep[Proposition 10.21]{rockafellar2009variational}. Define the solution map $\Gamma:\mathbb{R}^m\times \mathbb{R}^n\rightarrow \mathbb{R}^n$ as $\Gamma(y,g):=\{x\in \mathbb{R}^n\mid Ax=y,-g\in \partial p(x)\}$. Note that $\Gamma$ is a polyhedral multifunction, thus it is locally upper Lipschitz continuous at any $(y,g)\in \mathbb{R}^m\times \mathbb{R}^n$ by \citep{robinson1981some}. Therefore the desired conclusion holds by \citep[Corollary 1]{zhou2017unified}.
\end{proof}

Based on Proposition \ref{prop:luotseng}, we then prove that the error bound condition \eqref{error_bound} holds for the linear regression problem and the logistic regression problem with a piecewise linear-quadratic regularizer.
\begin{proposition} \label{prop:error-bound}
	Assume that $p(\cdot)$ is piecewise linear-quadratic. Then the error bound condition \eqref{error_bound} holds if $h(\cdot)$ is strongly convex on any compact convex set in $\mathbb{R}^m$. In particular, the latter property is satisfied by the following two special cases:
	\begin{itemize}
		\item[(1)] (linear regression) $h(y)=\sum_{i=1}^m (y_i-b_i)^2/2$, for some given vector $b\in \mathbb{R}^m$;
		\item[(2)] (logistic regression) $h(y) = \sum_{i=1}^m \log(1 + \exp(-b_i y_i))$, for some given vector $b\in \{-1,1\}^m$.
	\end{itemize}
\end{proposition}
\begin{proof}
	Let $r>0$ be given. Due to the compactness of $\Omega$, the set $\{x\in \mathbb{R}^n\mid {\rm dist}(x,\Omega)\leq r\}$ is also compact and thus $\xi:=\max_{\{x: {\rm dist}(x,\Omega)\leq r\}}f(x)$ is finite. Due to Proposition \ref{prop:luotseng}, for this $\xi$, there exist constants $\kappa,\varepsilon>0$ such that
	\begin{align}
		{\rm dist}(x,\Omega)\leq \kappa \|R(x)\|\mbox{ \quad  for all $x\in\mathbb{R}^n$ with $f(x)\leq \xi$, $\|R(x)\|\leq \varepsilon$}.\label{eq:luocondition}
	\end{align}
	For any $x$ such that ${\rm dist}(x,\Omega)\leq r$, if $\|R(x)\|\leq \varepsilon$, from \eqref{eq:luocondition}, we have ${\rm dist}(x,\Omega)\leq \kappa \|R(x)\|$; if $\|R(x)\|> \varepsilon$, we have ${\rm dist}(x,\Omega)\leq r= (r/\varepsilon)\varepsilon < (r/\varepsilon) \|R(x)\|$. Therefore, it holds that
	\[
	{\rm dist}(x,\Omega)\leq \max\{\kappa,(r/\varepsilon)\} \|R(x)\|,\quad \forall x\in \mathbb{R}^n \mbox{ satisfying } {\rm dist}(x,\Omega)\leq r.
	\]
	Next, consider an arbitrary $x\in \mathbb{R}^n$ such that ${\rm dist}(x,\Omega)\leq r$. For any $y\in {\cal T}_f(x)$, we have that $x={\rm Prox}_{\lambda p}(x+y-A^T \nabla h(Ax)+c)$, and
	\begin{align*}
		\|R(x)\|=\|{\rm Prox}_{\lambda p}(x+y-A^T \nabla h(Ax)+c)-{\rm Prox}_{\lambda p}(x-A^T \nabla h(A x)+c)\|\leq \|y\|.
	\end{align*}
	Therefore, we have ${\rm dist}(x,\Omega)\leq \max\{\kappa,(r/\varepsilon)\}
	\|y\|$ for any $y\in {\cal T}_f(x)$. This implies that
	\begin{align*}
		{\rm dist}(x,\Omega)\leq \max\{\kappa,(r/\varepsilon)\}
		{\rm dist}(0,{\cal T}_f(x)).
	\end{align*}
	Since $x$ is arbitrarily chosen, the above inequality implies that the error bound condition \eqref{error_bound} holds.
\end{proof}

From Proposition \ref{prop:error-bound} and Theorem \ref{thm:convergence_ALM}, the preconditioned PPA, with large parameters $\{\sigma_k\}_{k=0}^{\infty}$, for solving the linear regression and logistic regression problems with the exclusive lasso regularizer is guaranteed to have fast linear convergence rate. When solving the weighted exclusive lasso models with the preconditioned PPA presented in Algorithm \ref{alg:ppdna}, it is clear that we need an efficient way to compute the proximal mapping of the weighted exclusive lasso regularizer.

\section{The proximal mapping of the weighted exclusive lasso regularizer and its generalized Jacobian}
\label{sec:proxJacobian}
We give a systematic study of the weighted exclusive lasso regularizer. Specifically, we derive an $O(n\log n)$ procedure to compute the proximal mapping ${\rm Prox}_{p}(\cdot)$ with $p(\cdot)=\Delta^{\mathcal{G},w}(\cdot)$, and characterize the corresponding generalized Jacobian. By the definition of $\Delta^{\mathcal{G},w}(\cdot)$, it is important for us to study ${\rm Prox}_{\rho\|w\circ\cdot\|_1^2}(a)$ for any $a\in \mathbb{R}^t$, where $w\in \mathbb{R}^t_{++}$ is a given weight vector and $\rho>0$ is a given scalar.

\subsection{An $O(t\log t)$ procedure to compute ${\rm Prox}_{\rho\|w\circ\cdot\|_1^2}(\cdot)$}
The following proposition will be useful in the subsequent analysis.
\begin{proposition}\label{prop: sign}
	Given $a\in \mathbb{R}^t$. For each $i\in \{1,2,\cdots,t\}$, we have $\big({\rm Prox}_{\rho\|w\circ\cdot\|_1^2}(a)\big)_i = 0$ if $a_i = 0$; $\big({\rm Prox}_{\rho\|w\circ\cdot\|_1^2}(a)\big)_i\geq 0$ if $a_i>0$; and $\big({\rm Prox}_{\rho\|w\circ\cdot\|_1^2}(a)\big)_i\leq 0$ if $a_i<0$.
\end{proposition}	
\begin{proof}
	For notational simplicity, we denote
	\begin{align*}
		z^* = {\rm Prox}_{\rho\|w\circ\cdot\|_1^2}(a) = \underset{z\in \mathbb{R}^t}{\arg\min} \Big\{
		\eta(z):=\frac{1}{2}\|z-a\|^2 + \rho \|w\circ z\|_1^2
		\Big\}.
	\end{align*}
	First, consider the case when $a_i= 0$. We prove by contradiction. Suppose $z^*_i\neq 0$, then we define a new vector $\hat{z}\in \mathbb{R}^t$ as $\hat{z}_i = -z^*_i$ and $\hat{z}_j = z^*_j$ for $j\neq i$. By definition, we have $\hat{z} \neq z^*$ and $\rho \|w\circ \hat{z}\| = \rho \|w\circ z^*\|$. Moreover, we can see that
	\begin{align*}
		\eta(\hat{z})-\eta(z^*) = \frac{1}{2}\|\hat{z}-a\|^2 + \rho \|w\circ \hat{z}\|_1^2 - \frac{1}{2}\|z^*-a\|^2 - \rho \|w\circ z^*\|_1^2  = -2\hat{z}_i a_i+ 2z^*_i a_i = 4z^*_i a_i = 0,
	\end{align*}
	which contradicts the fact that $z^*$ is the unique minimizer of $\eta(\cdot)$. This implies that $z^*_i= 0$. Next we consider the case when $a_i> 0$. Again, we prove by contradiction. Assume $z^*_i < 0$, then we define $\hat{z}\in \mathbb{R}^t$ as $\hat{z}_i = -z^*_i$ and $\hat{z}_j = z^*_j$ for $j\neq i$. It can be see that $\eta(\hat{z})-\eta(z^*) =4z^*_i a_i <0$, which contradicts the fact that $z^*$ minimizes the function $\eta(\cdot)$ and further means that $z^*_i\geq 0$. The case when $a_i< 0$ can be proved in the same manner.
\end{proof}

The following proposition indicates that we only need to focus on computing ${\rm Prox}_{\rho\|w\circ\cdot\|_1^2}(|a|)$ for any $a \in \mathbb{R}^t$.

\begin{proposition}\label{prop:proxmapping}
	For given $\rho > 0$ and $a \in \mathbb{R}^{t}$, we have
	\[
	{\rm Prox}_{\rho\|w\circ\cdot\|_1^2}(a)={\rm sign}(a)\circ {\rm Prox}_{\rho\|w\circ\cdot\|_1^2}(|a|)={\rm sign}(a)\circ x(|a|),
	\]
	where $x(\cdot): \mathbb{R}^t_+ \to \mathbb{R}^t_+$ is defined as:
\begin{align}
	x(d) &:=\underset{x \in \mathbb{R}^n_{+}}{\arg\min}\ \Big\{\frac{1}{2}\|x - d\|^2 + \rho\|w\circ x\|_1^2
	\Big\}=\underset{x \in \mathbb{R}^n_{+}}{\arg\min} \ \Big\{\frac{1}{2}\|x - d\|^2 +  \rho x^T(ww^T)x \Big\}.\label{eq:pos_problem}
\end{align}
\end{proposition}
\begin{proof}
Let $s_a\in\mathbb{R}^t $ be defined as $(s_a)_i =1$ if $a_i\geq 0$, and $(s_a)_i = -1$ if $a_i<0$.
Note that
\begin{align*}
	\frac{1}{2}\|x-a\|^2 + \rho \|w\circ x\|_1^2 &= \frac{1}{2}\|s_a\circ x-|a|\|^2 + \rho \|w\circ x\|_1^2 = \frac{1}{2}\|s_a\circ x-|a|\|^2 + \rho \|w\circ s_a\circ x\|_1^2,
\end{align*}
where the first equality holds as $|a|= s_a\circ a$ and the second equality follows from the fact that $\|w\circ \cdot\|_1^2$ is invariant to sign changes. Therefore, we have that
\begin{align*}
	{\rm Prox}_{\rho\|w\circ\cdot\|_1^2}(a) = s_a \circ {\rm Prox}_{\rho\|w\circ\cdot\|_1^2}(|a|).
\end{align*}
It follows from Proposition \ref{prop: sign} that ${\rm Prox}_{\rho\|w\circ\cdot\|_1^2}(|a|) = x(|a|)$ and $x(|a|)_i = 0$ if $a_i = 0$. Thus, the conclusion holds.
\end{proof}
The next proposition provides an explicit formula for computing $x(\cdot)$ for any $d\in \mathbb{R}^t_{+}$. Since $x(d)_i = 0$ if $d_i = 0$, it is sufficient to consider that $d \in \mathbb{R}^t_{++}$.
\begin{proposition}\label{prop:pos_prox}
	Given $\rho > 0$ and $d \in \mathbb{R}_{++}^t$. Let $d^{w}\in \mathbb{R}^t$ be defined as $d^{w}_i :=d_i/w_i$, for $i=1,\cdots,{t}$. There exists a permutation matrix $\Pi$ such that $\Pi d^{w}$ is sorted in a non-increasing order. Denote $\tilde{d}=\Pi d $,  $\tilde{w}=\Pi w $, and
	\begin{align*}
	s_{i} = \sum_{j=1}^{i}
	\tilde{w}_j \tilde{d}_j, \quad
	L_{i} = \sum_{j=1}^{i}\tilde{w}_j^2,\quad
	\alpha_i = \frac{  s_{i}}{1 + 2\rho L_{i}}, \quad i = 1, 2, \dots, t.
	\end{align*}
	Let $\bar{\alpha}=\max_{1\leq i \leq t}\alpha_i$. Then, $x(d)$ in \eqref{eq:pos_problem} can be computed analytically as $x(d) = (d -2\rho \bar{\alpha} w)^+$.
\end{proposition}
\begin{proof}
	The KKT conditions for \eqref{eq:pos_problem} are given by
	\begin{equation}\label{eq: pos_KKT}
	x - d + 2\rho w w^Tx + \mu = 0,\;
	\mu \circ x =0,\; \mu \leq 0, \;x \geq 0,
	\end{equation}
	where $\mu \in \mathbb{R}^{t}$ is the corresponding dual multiplier. If $(x^*,\mu^{*})$ satisfies the KKT conditions \eqref{eq: pos_KKT},  by denoting $\beta=w^T x^*$, we can see that
	\begin{equation*}
	x^* +  \mu^* = d - 2\rho\beta w ,\;
	\mu^* \circ x^* =0,\; \mu^* \leq 0, \;x^* \geq 0.
	\end{equation*}
	Therefore, $(x^*,\mu^{*})$ have the representations:
	\[
	x^*=(d - 2\rho\beta w)^+,\quad \mu^*=(d - 2\rho\beta w)^-.
	\]
	Then our aim is to find the value of $\beta$. By the definition of $\beta$, we can see that
	\begin{align*}
	\beta = \sum_{i=1}^{t} w_i x^*_i=\sum_{i=1}^{t} w_i (d_i - 2\rho\beta w_i)^+=\sum_{i=1}^{t} w^2_i ((d^{w})_i - 2\rho\beta )^+=\sum_{i=1}^{t} \tilde{w}^2_i ((\Pi d^{w})_i - 2\rho\beta )^+.
	\end{align*}
	Note that there must exist some index $j$ such that $(\Pi  d^{w})_j>2\rho \beta$, otherwise, we have $\beta=0$ and $\Pi d^{w}\leq 0$ (equivalent to $d\leq 0$), which contradicts the assumption that $0\not=d \geq 0$. Since $\Pi d^{w}$ is sorted in a non-increasing order, there exists an index $k$ such that $\tilde{d}_{1}/\tilde{w}_{1} \geq \dots \geq
	\tilde{d}_{k}/\tilde{w}_{k} > 2\rho\beta \geq \tilde{d}_{k+1}/\tilde{w}_{k+1} \geq \dots \geq \tilde{d}_{t}/\tilde{w}_{t}$. Therefore,
	\[
	\beta = \sum_{i=1}^{k} \tilde{w}^2_i ((\Pi d^{w})_i - 2\rho\beta )=\sum_{i=1}^{k}  \tilde{w}_i\tilde{d}_i- 2\rho\beta \sum_{i=1}^{k}  \tilde{w}^2_i=s_{k} -2\rho \beta L_{k} ,
	\]
	which means that
	\[
	\beta = \frac{s_{k} }{1+2\rho L_{k} }=\alpha_k.
	\]
	Next we show that $\beta=\bar{\alpha}$, which means $\alpha_k\geq \alpha_i$ for all $i$. For $i<k$,
	\begin{align*}
	&\ \alpha_k - \alpha_{i} =  \frac{(1 + 2\rho L_{i}) s_{k} - (1+2\rho L_{k}) s_{i}}{(1+2\rho L_{k})(1+2\rho L_{i})}=  \frac{(1 + 2\rho L_{k})(s_k-s_i) -2\rho s_k\sum_{j=i+1}^k\tilde{w}_j^2}{(1+2\rho L_{k})(1+2\rho L_{i})}\\
	&=  \frac{(1 + 2\rho L_{k})\sum_{j=i+1}^k\tilde{w}_j\tilde{d}_j -2\rho (1 + 2\rho L_{k})\beta\sum_{j=i+1}^k\tilde{w}_j^2}{(1+2\rho L_{k})(1+2\rho L_{i})}=  \frac{\sum_{j=i+1}^k \tilde{w}_j^2 (\tilde{d}_j/\tilde{w}_j -2\rho \beta)}{1+2\rho L_{i}}\geq 0.
	\end{align*}
	We can prove that $\alpha_k\geq \alpha_i$ for all $i>k$ in a similar way. Therefore, we have that $\beta=\alpha_k=\max_{1\leq i \leq {t}}\alpha_i=\bar{\alpha}$. Finally, since the solution to \eqref{eq:pos_problem} is unique, we have
	\begin{equation*}
	x(a)=x^*=(d - 2\rho\beta w)^+=(d - 2\rho\bar{\alpha} w)^+.\qedhere
	\end{equation*}
\end{proof}

As a side note, the computational cost of calculating $x(d)$ given any $d\in \mathbb{R}^t_+$ is $O(t\log t)$, where the most time-consuming step is to sort a $t$-dimensional vector. Thus, the computational cost of calculating ${\rm Prox}_{\rho\|w\circ\cdot\|_1^2}(a)$ is $O(t\log t)$ due to Proposition \ref{prop:proxmapping}.

\begin{remark}
	The proximal mapping of $\rho\|w\circ\cdot\|_1^2$ is mentioned in \citep[Proposition 4]{kowalski2009sparse}, but we find that the derivation contains some errors. We use a simple example to demonstrate the gap. Consider the proximal mapping of $\|\cdot\|_1^2$ at the point $[1,0.5]^T$, where we need to solve
	\begin{align}
	\label{eq: counterexample}
	\min_{x_{1},x_{2}\in\mathbb{R}}\Big\{\phi(x_1,x_2):=\frac{1}{2}(x_{1}-1)^2+\frac{1}{2}(x_{2}-0.5)^2+(|x_{1}|+|x_{2}|)^2\Big\}.
	\end{align}
	Equation (21) in \citep{kowalski2009sparse} said that
	\begin{align*}
	|x_{1}| = 1-2(|x_{1}|+|x_{2}|),\quad
	|x_{2}| = 0.5-2(|x_{1}|+|x_{2}|),
	\end{align*}
	that is $|x_{1}|=2/5$, $|x_{2}|=-1/10$, which contradicts the fact that $|x_{2}|$ should be nonnegative.
	
	Later in 2017, the authors of \citep{yoon2017combined} also proposed a formula for the proximal mapping of $\|\cdot\|_1^2$, which is not correct. The problem \eqref{eq: counterexample} can also serve as a counterexample. The solution obtained by the formula in \citep[Equation (8)]{yoon2017combined} is $x_{1} = x_{2} = 0$, which is not optimal to \eqref{eq: counterexample} since 
	\[
	\phi(0,0) = 5/8 > \phi(1/3,0)=11/24.
	\]
\end{remark}

\subsection{The generalized Jacobian of ${\rm Prox}_{\rho\|w\circ\cdot\|_1^2}(\cdot)$}
Note that in order to design a second-order type algorithm for solving the dual of the PPA subproblem, it is critical for us to derive an explicit element in the generalized Jacobian of ${\rm Prox}_{\rho\|w\circ\cdot\|_1^2}(\cdot)$. According to Proposition \ref{prop:proxmapping}, we know that in order to obtain the generalized Jacobian of ${\rm Prox}_{\rho\|w\circ\cdot\|_1^2}(\cdot)$, we need to study the generalized Jacobian of $x(\cdot)$ first. For a better illustration, we consider the quadratic programming (QP) reformulation of $x(\cdot)$. For any $a\in \mathbb{R}^{t}$, if we denote $ Q := I_{t} +  2\rho ww^T \in \mathbb{S}^{t}$, then \eqref{eq:pos_problem} can be equivalently written as
\begin{align}
x(a)=\underset{x \in \mathbb{R}^{t}}{\arg\min} \ \Big\{ \frac{1}{2} \langle x, Qx \rangle - \langle x, a \rangle \mid x\geq 0 \Big\}.\label{eq: QP}
\end{align}

We first consider the case when $a\in \mathbb{R}^{t}$ satisfies $Q^{-1}a\geq 0$. In this case, we can derive that
\begin{align*}
x(a) = Q^{-1}a = \Big( I-\frac{2\rho}{1+2\rho w^T w}ww^T \Big)a.
\end{align*}
Therefore, $x(a)$ is differentiable on $\{a\in \mathbb{R}^{t}\mid Q^{-1}a>0\}$. Next we consider the case when there exists $i\in \{1,\cdots,{t}\}$ such that $(Q^{-1}a)_i<0$. Here, we derive the HS-Jacobian \citep{li2018efficiently} of $x(\cdot)$ based on the strongly convex QP \eqref{eq: QP} by applying the general results established in \citep{han1997newton,li2018efficiently}. As one can see from the KKT system \eqref{eq: pos_KKT} and the uniqueness of $x(a)$, the dual multiplier $\mu$ is also unique, which we denote as $\mu(a)$. Denote the active set of $x(a)$ as
\begin{align}
I(a): = {\{i \in \{ 1,\ldots,{t} \} \mid (x(a))_i = 0\}}.
\label{eq:Ia}
\end{align}
Since now we consider the case when there exists $i$ such that $(Q^{-1}a)_i<0$, we know that $\mu(a)\neq 0$, which implies that $I(a)\neq \emptyset$. Define a collection of index sets:
\[
{\cal K}(a): =  \{\;  K\subseteq \{1,\ldots,{t}\} \mid  {\rm supp}(\mu(a)) \subseteq K\subseteq I(a)\},
\]
where ${\rm supp}(\mu(a))$ denotes  the set of indices $i$ such that $\mu(a)_i \neq 0$. Note that the set ${\cal K}(a)$ is non-empty due to the complementarity condition $\mu(a)\circ x(a)=0$, $\mu(a)\leq 0$, $x(a)\geq 0$, and the fact that $I(a)\neq \emptyset$. Since the B-subdifferential $\partial_B x(a)$ is difficult to compute, we define the multifunction
\begin{equation*}
\partial_{\rm HS}x(a) := \left\{
P\in\mathbb{R}^{{t}\times {t}}\mid P = Q^{-1} - Q^{-1}
I_K^T\left( I_K Q^{-1}I_K^T\right)^{-1}I_K Q^{-1},\; K\in{\cal K}(a)
\right\},
\end{equation*}
as a computational replacement for $\partial_Bx(a)$, where $I_K$ is the matrix consisting of the rows of $I_{t}$, indexed by $K$. The set $\partial_{\rm HS}x(a)$ is the HS-Jacobian of $x(\cdot)$ at $a$ when $a\in \{a\mid \exists \ i \ {\rm s.t.}\  (Q^{-1}a)_i<0\}$. Combining the above two cases, we define the multifunction $\widehat{\partial}_{\rm HS}x(\cdot)$: $\mathbb{R}^{t} 	\rightrightarrows \mathbb{R}^{{t}\times {t}}$ as
\begin{align}
\widehat{\partial}_{\rm HS}x(a) \!\!=\!\! 	\left\{
\begin{aligned}
& Q^{-1} && \mbox{if $Q^{-1}a>0$}\\
&\!\!\left\{
Q^{-1} \!-\! Q^{-1}
I_K^T\left( I_K Q^{-1}I_K^T\right)^{-1} \!\!I_K Q^{-1}\mid K\!\in\! {\cal K}(a)
\right\} && \mbox{if $\exists  i  \ {\rm s.t.}\  (Q^{-1}a)_i\!<\! 0$}\\
&Q^{-1} \cup \left\{
Q^{-1} - Q^{-1}
I_K^T\left( I_K Q^{-1}I_K^T\right)^{-1}I_K Q^{-1}\mid K\in{\cal K}(a)
\right\} && \mbox{otherwise}
\end{aligned}\right.  \label{eq: xpartial}
\end{align}
for any $a\in \mathbb{R}^{t}$, which can be regarded as a generalized Jacobian of $x(a)$.

Define the multifunction $\partial_{\rm HS} {\rm Prox}_{\rho\|w\circ\cdot\|_1^2}:\mathbb{R}^{t}\rightrightarrows \mathbb{R}^{{t}\times {t}}$ by
\begin{align}
\partial_{\rm HS} {\rm Prox}_{\rho\|w\circ\cdot\|_1^2}(a)=\left\{
{\rm Diag}(\theta ) P {\rm Diag}(\theta ) \mid \theta\in {\rm SGN}(a),\ P \in \widehat{\partial}_{\rm HS}x(|a|)
\right\}, \; \forall\; a\in \mathbb{R}^{t},
\label{eq: Jacobian_prox_12}
\end{align}
where $\widehat{\partial}_{\rm HS}x(\cdot)$ is defined in \eqref{eq: xpartial}, and ${\rm SGN}:\mathbb{R}^{t}\rightrightarrows \mathbb{R}^{t}$ is defined as
\begin{align*}
{\rm SGN}(z):=\left\{u\in \mathbb{R}^{t}:u_j\in \left\{\begin{aligned}
&\{ {\rm{sign}}(z_j)\} && \mbox{if}\ z_j \not= 0 \\
&[-1,1] && \mbox{if}\ z_j=0\\
\end{aligned}\right.,\ j = 1,\cdots,{t}\right\}.
\end{align*}

The next proposition states the reason why we can treat $\partial_{\rm HS} {\rm Prox}_{\rho\|w\circ\cdot\|_1^2}(a)$ as the generalized Jacobian of ${\rm Prox}_{\rho\|w\circ\cdot\|_1^2}(\cdot)$ at $a$.
\begin{proposition}\label{prop: property_HS_prox}
	$\partial_{\rm HS} {\rm Prox}_{\rho\|w\circ\cdot\|_1^2}(\cdot)$ is a nonempty, compact valued and upper-semicontinuous multifunction. For any $a\in \mathbb{R}^{t}$, the elements in $\partial_{\rm HS} {\rm Prox}_{\rho\|w\circ\cdot\|_1^2}(a)$ are all symmetric and positive semidefinite. Moreover, ${\rm Prox}_{\rho\|w\circ\cdot\|_1^2}(\cdot)$ is strongly semismooth with respect to $\partial_{\rm HS} {\rm Prox}_{\rho\|w\circ\cdot\|_1^2}(\cdot)$.
\end{proposition}
\begin{proof}
	By the definition of $\partial_{\rm HS} {\rm Prox}_{\rho\|w\circ\cdot\|_1^2}(\cdot)$, we can see that it is a nonempty and compact valued multifunction. Fix $a\in \mathbb{R}^{t}$. The symmetry of the elements in $\partial_{\rm HS} {\rm Prox}_{\rho\|w\circ\cdot\|_1^2}(a)$ follows naturally by the definition in \eqref{eq: Jacobian_prox_12}. In order to prove that the elements in $\partial_{\rm HS} {\rm Prox}_{\rho\|w\circ\cdot\|_1^2}(a)$ are all positive semidefinite, it suffices to prove that the elements in $\widehat{\partial}_{\rm HS}x(a)$ are all positive semidefinite. The case when $Q^{-1}a>0$ is obvious from the definition of $Q$. Thus we only need to consider the case when there exists $i$ such that $(Q^{-1}a)_i<0$. For any $K\in {\cal K}(a)$, denote $\xi\in \mathbb{R}^{t}$ with $\xi_i = 0$ if $i\in K$, and $\xi_i = 1$ otherwise. Let $\Xi = I_{t}-{\rm Diag}(\xi)$. After some algebraic multiplications, we can see that
	\begin{align*}
	I_K^T\left( I_K Q^{-1}I_K^T\right)^{-1}I_K=(\Xi Q^{-1}\Xi)^{\dagger}=\Xi (\Xi Q^{-1}\Xi)^{\dagger}\Xi,
	\end{align*}
	where the last equality follows from the fact that $\Xi$ is a $0$-$1$ diagonal matrix. Then by \citep[Proposition 3]{li2018efficiently}, we have
	\begin{align*}
	Q^{-1} - Q^{-1}I_K^T\left( I_K Q^{-1}I_K^T\right)^{-1}I_K Q^{-1}=Q^{-1} - Q^{-1}\Xi (\Xi Q^{-1}\Xi)^{\dagger}\Xi Q^{-1}= ({\rm Diag}(\xi) Q {\rm Diag}(\xi))^{\dagger}\succeq 0,
	\end{align*}
	which further implies that the elements in $\partial_{\rm HS} {\rm Prox}_{\rho\|w\circ\cdot\|_1^2}(a)$ are all positive semidefinite.
	
	According to \citep[Proposition 2]{li2018efficiently} and the definition of $\widehat{\partial}_{\rm HS}x(\cdot)$ in \eqref{eq: xpartial} , we know that for the given $a$, there exists a neighborhood $U$ of $a$ such that for any $a' \in U$, $\widehat{\partial}_{\rm HS}x(|a'|)\subseteq \widehat{\partial}_{\rm HS}x(|a|)$ and
	\begin{align}
	x(|a'|) - x(|a|) - P(|a'| - |a|)=0,\quad \forall P \in \widehat{\partial}_{\rm HS}x(|a'|).\label{eq: strongly_xa}
	\end{align}
	By the definition of ${\rm SGN}(\cdot)$, if we take the neighbourhood $U$ to be sufficiently small, then we have ${\rm SGN}(a')\subseteq {\rm SGN}(a)$ for any $a'\in U$. Therefore, it holds that $\partial_{\rm HS} {\rm Prox}_{\rho\|w\circ\cdot\|_1^2}(a')\subseteq \partial_{\rm HS} {\rm Prox}_{\rho\|w\circ\cdot\|_1^2}(a)$ for all $a'\in U$, which implies that $\partial_{\rm HS} {\rm Prox}_{\rho\|w\circ\cdot\|_1^2}$ is upper-semicontinuous at $a$. Since ${\rm Prox}_{\rho\|w\circ\cdot\|_1^2}(\cdot)$ is piecewise linear and Lipschitz continuous, it is directionally differentiable according to \citep{facchinei2007finite}. Note that for all $a'\in U$, since ${\rm SGN}(a')\subseteq {\rm SGN}(a)$, we have ${\rm Diag}(\theta) (a'-a)=|a'|-|a|$ with any $\theta \in {\rm SGN}(a')$. Therefore, from \eqref{eq: strongly_xa}, it holds that for any $a'\in U$,
	\begin{align*}
	\theta \circ x(|a'|) - \theta \circ x(|a|) - {\rm Diag}(\theta )P{\rm Diag}(\theta )(a'-a)=0,\quad \forall
	\theta \in {\rm SGN}(a'),\ \forall P \in \widehat{\partial}_{\rm HS}x(|a'|).
	\end{align*}
	By Proposition \ref{prop: sign}, for any $i$, if $(x(|a'|))_i\neq 0$, then we must have $a'_i\neq 0$, which further implies $\theta_i={\rm sign}(a'_i)$ for each $\theta\in {\rm SGN}(a')$. Therefore, we know that for all $a'\in U$,
	\begin{align*}
	\theta \circ x(|a'|) = {\rm sign}(a')\circ x(|a'|),\quad \theta \circ x(|a|) = {\rm sign}(a)\circ x(|a|),\quad \forall \theta \in {\rm SGN}(a')\subseteq {\rm SGN}(a).
	\end{align*}
	That is to say, when $a'\in U$,
	\begin{align*}
	{\rm Prox}_{\rho\|w\circ\cdot\|_1^2}(a')-{\rm Prox}_{\rho\|w\circ\cdot\|_1^2}(a)-M(a'-a)=0,\quad \forall M\in \partial_{\rm HS} {\rm Prox}_{\rho\|w\circ\cdot\|_1^2}(a').
	\end{align*}
	Thus ${\rm Prox}_{\rho\|w\circ\cdot\|_1^2}(\cdot)$ is strongly semismooth with respect to $\partial_{\rm HS} {\rm Prox}_{\rho\|w\circ\cdot\|_1^2}(\cdot)$ at $a$.
\end{proof}

In practice, we always need a specific element in $\partial_{\rm HS} {\rm Prox}_{\rho\|w\circ\cdot\|_1^2}(a)$ at any $a \in \mathbb{R}^{t}$. In the following proposition, we provide a highly efficient way to construct one specific element in $\partial_{\rm HS} {\rm Prox}_{\rho\|w\circ\cdot\|_1^2}(a)$, which is a $0$-$1$ diagonal matrix plus a rank-one correction.
\begin{proposition}\label{compute_M}
	Given $a \in \mathbb{R}^{t}$. The following properties hold.
	\begin{enumerate}
		\item If $Q^{-1}a\geq 0$, we have that $Q^{-1}=I-\frac{2\rho}{1+2\rho w^T w}ww^T$ is an element in $\partial_{\rm HS} {\rm Prox}_{\rho\|w\circ\cdot\|_1^2}(a)$;
		\item If there exists $i\in \{1,\cdots,{t}\}$ such that $(Q^{-1}a)_i<0$, denote
		\begin{align*}
		P_0 :=  Q^{-1} - Q^{-1}I_{I(|a|)}^T
		(I_{I(|a|)} Q^{-1}
		I_{I(|a|)}^T)^{-1}
		I_{I(|a|)}Q^{-1},
		\end{align*}
		where $I(\cdot)$ is defined in \eqref{eq:Ia}, the matrix
		\begin{equation}\label{eq:M0}
		M_0 := {\rm Diag}({\rm sign}(a)) P_0  {\rm Diag}({\rm sign}(a))
		\end{equation}
		is an element in the set $\partial_{\rm HS} {\rm Prox}_{\rho\|w\circ\cdot\|_1^2}(a)$. Moreover, if we define $\xi\in \mathbb{R}^{t}$ with $\xi_i = 0$ when $i\in I(|a|)$ but $\xi_i = 1$ otherwise, and $\tilde{w} :=({\rm sign}(a)\circ \xi)\circ w$, then $M_0$ defined in \eqref{eq:M0} can be computed as
		\begin{equation*}
		M_0 = {\rm Diag}(\xi)-\frac{2\rho}{1+2\rho (\tilde{w}^T\tilde{w})} \tilde{w}\tilde{w}^T.
		\end{equation*}
	\end{enumerate}
\end{proposition}
\begin{proof}
	(1) and the first part of (2) follow immediately from the definition of $\partial_{\rm HS} {\rm Prox}_{\rho\|w\circ\cdot\|_1^2}(\cdot)$. Suppose there exists $i$ such that $(Q^{-1}a)_i<0$. Similarly to the proof in Proposition \ref{prop: property_HS_prox}, we have
	\begin{align*}
	P_0 = ({\rm Diag}(\xi) Q {\rm Diag}(\xi))^{\dagger}.
	\end{align*}
	Denote $\hat{w}=\xi \circ w$. Since $Q = I_{t} +  2\rho ww^T \in \mathbb{R}^{t}\times {t}$, it holds that
	\begin{align*}
	P_0 &= ({\rm Diag}(\xi) Q {\rm Diag}(\xi))^{\dagger}
	= ({\rm Diag}(\xi) + 2\rho \hat{w}\hat{w}^T)^{\dagger}={\rm Diag}(\xi) -
	\frac{2\rho}{1+2\rho (\hat{w}^T\hat{w})} \hat{w}\hat{w}^T.
	\end{align*}
	Next we show that ${\rm sign}(a)\circ{\rm sign}(a)\circ\xi = \xi$. First, we note that $a_i=0$ implies that $(x(|a|))_i=0$, and hence $\xi_i=0$. Thus if $a_i =0$, then ${\rm sign}(a_i)^2 \xi_i = 0 = \xi_i$. For the case when $a_i\not=0$, we clearly have ${\rm sign}(a_i)^2 \xi_i = \xi_i$. Similarly, we can prove that $\hat{w}^T\hat{w}=\tilde{w}^T\tilde{w}$. Now it is easy to see that
	\begin{align*}
	M_0 &=   {\rm Diag}({\rm sign}(a)) \Big({\rm Diag}(\xi)-
	\frac{2\rho}{1+2\rho (\hat{w}^T\hat{w})} \hat{w}\hat{w}^T\Big)  {\rm Diag}({\rm sign}(a)) = {\rm Diag}(\xi) -
	\frac{2\rho}{1+2\rho (\tilde{w}^T\tilde{w})} \tilde{w}\tilde{w}^T.
	\end{align*}
	This completes the proof.
\end{proof}

\subsection{The proximal mapping of $\Delta^{\mathcal{G},w}(\cdot)$ and its generalized Jacobian}
Based on the equality \eqref{eq: reformulation_exclusive} and the previous discussions, we summarize the following proposition, which gives the proximal mapping of $\Delta^{\mathcal{G},w}(\cdot)$ and its corresponding generalized Jacobian.
\begin{proposition}\label{prop: proxmapping_and_jacobian}
	Given $\nu>0$, and $p(\cdot)=\Delta^{\mathcal{G},w}(\cdot)$. The following statements hold.
	
	\begin{enumerate}
		\item The proximal mapping ${\rm Prox}_{\nu p}(\cdot)$ can be computed as
	\begin{align*}
	{\rm Prox}_{\nu p}(x)&={\cal P}^T
	\underset{y\in \mathbb{R}^n}{\arg\min}\  \Big\{
	\frac{1}{2}\|y-{\cal P}x\|^2+\nu \sum_{j=1}^l \|({\cal P} w)^{(j)} \circ y^{(j)}\|_1^2\Big\}\\
	&={\cal P}^T \big[{\rm Prox}_{\nu\|({\cal P} w)^{(1)}\circ\cdot\|_1^2}(({\cal P} x)^{(1)});\cdots;{\rm Prox}_{\nu\|({\cal P} w)^{(l)}\circ\cdot\|_1^2}(({\cal P} x)^{(l)})\big],
	\end{align*}
	where ${\rm Prox}_{\nu\|({\cal P} w)^{(j)}\circ\cdot\|_1^2}(\cdot)$, for each $j=1,\cdots,l$, is defined in Proposition \ref{prop:proxmapping}.
	
	\item Define the multifunction $\partial_{\rm HS}{\rm Prox}_{\nu p}: \mathbb{R}^{t}\rightrightarrows \mathbb{R}^{{t}\times {t}}$ as
	\begin{align*}
	\partial_{\rm HS}{\rm Prox}_{\nu p}(x)=\left\{ {\cal P}^T {\rm Diag}(M_1,\cdots,M_l){\cal P}\mid M_j\in \partial_{\rm HS} {\rm Prox}_{\nu\|({\cal P} w)^{(j)}\circ\cdot\|_1^2}(({\cal P} x)^{(j)}), j = 1,\cdots,l	\right\},
	\end{align*}
	where $\partial_{\rm HS} {\rm Prox}_{\nu\|({\cal P} w)^{(j)}\circ\cdot\|_1^2}(\cdot)$, for each $j = 1,\cdots,l$, is defined in \eqref{eq: Jacobian_prox_12}. Then $\partial_{\rm HS}{\rm Prox}_{\nu p}(\cdot)$ can be regarded as the generalized Jacobian of ${\rm Prox}_{\nu p}(\cdot)$ satisfying the following properties.
	\begin{itemize}
		\item[(a)] $\partial_{\rm HS}{\rm Prox}_{\nu p}(\cdot)$ is a nonempty, compact valued and upper-semicontinuous multifunction;
		\item[(b)] for any $x\in \mathbb{R}^n$, the elements in $\partial_{\rm HS}{\rm Prox}_{\nu p}(x)$ are symmetric and positive semidefinite;
		\item[(c)] ${\rm Prox}_{\nu p}(\cdot)$ is strongly semismooth with respect to $\partial_{\rm HS}{\rm Prox}_{\nu p}(\cdot)$.
	\end{itemize}
	In addition, we can construct a specific element in $\partial_{\rm HS}{\rm Prox}_{\nu p}(x)$ according to Proposition \ref{compute_M}.
	\end{enumerate}
\end{proposition}

\section{A semismooth Newton method for solving the dual of the PPA subproblem}
\label{sec:ssn}
Note that the key challenge in executing the preconditioned PPA is whether the dual of the subproblem can be solved efficiently. We will design a highly efficient second-order type algorithm, which is expected to be superlinearly (or even quadratically) convergent. Before going into detail of the algorithm design, we first present the following proposition to show the strict concavity of the function $\psi_k(\cdot)$, which ensures that the problem \eqref{eq:ATA_D} admits a unique maximizer.

\begin{proposition}\label{prop: jacobian_h}
	Suppose $h(\cdot)$ is convex and twice continuously differentiable. For any $\nu>0$, ${\rm Prox}_{\nu h}(z)$ is differentiable with
	\begin{align*}
		\nabla {\rm Prox}_{\nu h}(z)=(I_m+\nu \nabla^2 h({\rm Prox}_{\nu h}(z)))^{-1},\quad \forall z\in \mathbb{R}^m.
	\end{align*}
	Therefore, $0 \prec \nabla {\rm Prox}_{\nu h}(z) \preceq I_m$ for any $z\in \mathbb{R}^m$, and  $\theta(z):=\|z\|^2/2-{\rm E}_{\nu h}(z)$ is strictly convex.
\end{proposition}
\begin{proof}
	Define $F:\mathbb{R}^{2m}\rightarrow \mathbb{R}^m$ as $F(u,v) = v-u+\nu \nabla h(v)$, for all $  (u,v)\in \mathbb{R}^m\times\mathbb{R}^m$.
	The optimality condition of $\min_w\{\frac{1}{2}\|w-z\|^2+\nu h(w)\}$ implies that for any $z\in \mathbb{R}^m$, there exists a unique $w$ such that $F(z,w)=0$, which is denoted as ${\rm Prox}_{\nu h}(z)$. Let the Jacobian of $F$ with respect to $u$ and $v$ be denoted as $J_{F,u}$ and $J_{F,v}$, respectively. We have that $	J_{F,v}(z,w)=I_m +\nu \nabla^2 h(w)$
	is invertible. According to the implicit function theorem, we know that there exists an open set $U\subset \mathbb{R}^m$ containing $z$ such that there exists a unique continuously differentiable function $g: U\rightarrow \mathbb{R}^m$ such that $g(z)=w$ and
	\begin{align*}
		&F(u,g(u))=g(u)-u+\nu \nabla h(g(u))=0,\quad \forall u\in U,\\
		&\nabla g(u)=-[J_{F,v}(u,g(u))]^{-1}J_{F,u}(u,g(u))=(I_m +\nu \nabla^2 h(g(u)))^{-1},\quad \forall u\in U.
	\end{align*}
	Combining the uniqueness of the function $g(\cdot)$ and the definition of ${\rm Prox}_{\nu h}(\cdot)$, we have that ${\rm Prox}_{\nu h}(u)=g(u)$ for all $u\in U$ and
	$\nabla {\rm Prox}_{\nu h}(z )= (I_m +\nu \nabla^2 h({\rm Prox}_{\nu h}(z)))^{-1}$.
	The remaining part of the conclusion follows naturally since $\nabla \theta (z)= {\rm Prox}_{\nu h}(z )$.
\end{proof}

\subsection{A semismooth Newton algorithmic framework}
Since $\psi_k$ is strictly concave and continuously differentiable, the unique maximizer can be computed by solving the equation
\begin{align}
	\nabla \psi_k(u)=-{\rm Prox}_{\sigma_k h/\tau}(Ax^k+\frac{\sigma_k}{\tau}u)+A {\rm Prox}_{\sigma_k \lambda p}(x^k+\sigma_k c-\sigma_kA^Tu)=0.\label{eq: newton-system}
\end{align}
Note that $\nabla \psi_k(\cdot)$ is Lipschitz continuous, but nondifferentiable. We propose a semismooth Newton (SSN) method to solve \eqref{eq: newton-system}. The concept of semismoothness can be found in \citep{kummer1988newton,mifflin1977semismooth,qi1993nonsmooth,sun2002semismooth}.

From now on, we consider the case when $p(\cdot)=\Delta^{{\cal G},w}(\cdot)$. Define the multifunction $\hat{\partial}^2 \psi_k(\cdot):\mathbb{R}^m\rightrightarrows \mathbb{R}^{m\times m}$ as follows: for any $u\in \mathbb{R}^m$,
\begin{align}
	\hat{\partial}^2 \psi_k(u)  :=  -\frac{\sigma_k}{\tau}\nabla {\rm Prox}_{\sigma_k h/\tau}(Ax^k+\frac{\sigma_k}{\tau}u)  - \sigma_kA\partial_{\rm HS}{\rm Prox}_{\sigma_k \lambda p}(x^k + \sigma_k c - \sigma_kA^Tu)A^T,
	\label{eq: jacobian}
\end{align}
where $\partial_{\rm HS} {\rm Prox}_{\sigma_k \lambda  p}(\cdot)$ is defined in Proposition \ref{prop: proxmapping_and_jacobian}. The following proposition states that $\hat{\partial}^2 \psi_k(\cdot)$ can be treated as the generalized Jacobian of $\nabla\psi_k(\cdot)$, which follows from Propositions \ref{prop: jacobian_h} and \ref{prop: proxmapping_and_jacobian}.
\begin{proposition}\label{prop: jacobian_psi}
	For the case when $p(\cdot)=\Delta^{{\cal G},w}(\cdot)$, the multifunction $\hat{\partial}^2 \psi_k(\cdot)$ defined in \eqref{eq: jacobian} satisfies the following properties:
	\begin{itemize}
		\item[(1)] $\hat{\partial}^2 \psi_k(\cdot)$ is a nonempty, compact valued, upper-semicontinuous multifunction;
		\item[(2)] for any $u\in \mathbb{R}^m$, all the elements in $\hat{\partial}^2 \psi_k(u) $ are symmetric and negative definite;
		\item[(3)] $\nabla \psi_k(\cdot)$ is strongly semismooth with respect to $\hat{\partial}^2 \psi_k(\cdot)$.
	\end{itemize}
\end{proposition}

Now we present the SSN method for solving \eqref{eq:ATA_D} in Algorithm \ref{alg:ssn}.
\begin{algorithm}[H]\small
	\caption{A semismooth Newton method for \eqref{eq:ATA_D}}
	\label{alg:ssn}
	\begin{algorithmic}[1]
		\STATE \textbf{Input}: $\mu \in (0, 1/2)$, $\bar{\tau} \in (0, 1]$, and $\bar{\gamma}, \delta \in (0, 1)$.
		\STATE \textbf{Output}: an approximate optimal solution $u^{k+1}$ to \eqref{eq:ATA_D}.
		\STATE \textbf{Initialization}: choose $u^{k,0} \in  \mathbb{R}^m$, $j=0$.
		\REPEAT
		\STATE {\bfseries Step 1}. Select an element ${\cal H}_j \in \hat{\partial}^{2} \psi_k(u^{k,j})$. Apply the direct method or the conjugate gradient (CG) method to find an approximate solution $d^j \in \mathbb{R}^m$ to
		\begin{align}\label{eq: cg-system}
			{\cal H}_j(d^j) \approx - \nabla \psi_k(u^{k,j}),
		\end{align}
		such that $\|{\cal H}_j(d^j) + \nabla\psi_k(u^{k,j})\| \leq \min(\bar{\gamma}, \|\nabla\psi(u^{k,j})\|^{1+\bar{\tau}})$.
		\\[3pt]
		\STATE {\bfseries Step 2}. Set $\alpha_j = \delta^{m_j}$, where $m_j$ is the smallest nonnegative integer $m$ for which
		$$\psi_k(u^{k,j} + \delta^m d^j) \geq \psi_k(u^{k,j}) + \mu\delta^m \langle \nabla\psi_k(u^{k,j}), d^j \rangle .$$
		\\[3pt]
		\STATE{\bfseries Step 3}. Set $u^{k,j+1} = u^{k,j} + \alpha_j d^j$, $u^{k+1}=u^{k,j+1}$, $j\leftarrow j+1$.
		\UNTIL{Stopping criterion based on $u^{k+1}$ is satisfied.}
	\end{algorithmic}
\end{algorithm}

In practice, one can choose the parameters in Algorithm \ref{alg:ssn} as $\mu = 10^{-4}$, $\bar{\tau}=0.5$, $\bar{\gamma} = 0.005$ and $\delta = 0.5$. The following theorem gives the convergence result of the SSN method, which can be proved by using Proposition \ref{prop: jacobian_psi} and the results in \citep[Proposition 3.3 and Theorem 3.4]{zhao2010newton}, \citep[Theorem 3]{li2018efficiently}. For simplicity, we omit the proof here.
\begin{theorem} \label{thm:convergence_SSN}
	Let $\{u^{k,j}\}$ be the sequence generated by Algorithm \ref{alg:ssn}. Then $\{u^{k,j}\}$ converges to the unique optimal solution $\bar{u}^{k+1}$ of the problem \eqref{eq:ATA_D}, and for $j$ sufficiently large,
	\begin{align*}
		\|u^{k,j+1} - \bar{u}^{k+1}\| = O(\|u^{k,j} - \bar{u}^{k+1}\|^{1+\bar{\tau}}),
	\end{align*}
	where $\bar{\tau} \in (0, 1]$ is given in the algorithm.
\end{theorem}

We should emphasize that the efficiency of computing the Newton direction in \eqref{eq: cg-system} depends critically on exploiting the sparsity structure of the generalized Jacobian. The practical implementation details are presented in the next subsection.

\subsection{Practical implementation of the SSN method}\label{sec: implement_SSN}
In the SSN method presented in Algorithm \ref{alg:ssn}, the key step is to compute the Newton direction, in other words, to solve the linear system \eqref{eq: cg-system}. In our implementation, we fully exploit the structured sparsity of the generalized Jacobian  which results in a highly efficient way to solve the linear system.

Denote $\tilde{A}:=A {\cal P}^T$. By the definition of the permutation matrix ${\cal P}$, $\tilde{A}$ can be obtained by permuting the columns in $A$ according to ${\cal P}$. Note that $\tilde{A}$ only needs to be computed once as a preprocessing step of the PPDNA algorithm since $\cal P$ is fully determined by the fixed group information $\cal G$. Given $(\tilde{x},\tilde{u})\in\mathbb{R}^n\times \mathbb{R}^m$ and $\sigma,\tau>0$, the Newton system \eqref{eq: cg-system} is in the form:
\begin{align}
\left( \frac{\sigma}{\tau}H  + \sigma \tilde{A} {\rm Diag}(M_1,\cdots,M_l)\tilde{A}^T\right) d = R,\label{eq: appendix_org}
\end{align}
where $R\in \mathbb{R}^m$ is a given vector, $H\in \nabla {\rm Prox}_{\sigma h/\tau}(A\tilde{x}+\frac{\sigma}{\tau}\tilde{u})$, $M_j\in \partial_{\rm HS} {\rm Prox}_{\sigma\lambda\|({\cal P} w)^{(j)}\circ\cdot\|_1^2}(({\cal P} \hat{x})^{(j)})$, $j=1,\cdots,l$, with $\hat{x}:=\tilde{x} + \sigma c - \sigma A^T\tilde{u}$. As shown in Proposition \ref{prop: jacobian_h}, $H$ is symmetric and positive definite. We denote the Cholesky decomposition of $H$ as $H=LL^T$, where $L$ is a nonsingular lower triangular matrix. Then we can reformulate the equation \eqref{eq: appendix_org} equivalently as
\begin{align*}
\left( \frac{\sigma}{\tau} I_m  + \sigma (L^{-1}\tilde{A})  {\rm Diag}(M_1,\cdots,M_l) (L^{-1}\tilde{A})^T\right) (L^T d) = L^{-1}R.
\end{align*}
Note that when we consider the linear regression or the logistic regression problems, the matrix $H$ is in fact a diagonal matrix, which means that we can compute $L$ and $L^{-1}$ with very low computational cost. For convenience, we write the linear system in a compact form as
\begin{align}
\left( I_m + \tau \hat{A} {\cal M} \hat{A}^T\right) \hat{d} = \hat{R},\label{eq: appendix_newton}
\end{align}
where $\hat{A}:= L^{-1}\tilde{A} \in \mathbb{R}^{m\times n}$, ${\cal M}:={\rm Diag}(M_1,\cdots,M_l)\in \mathbb{R}^{n\times n}$, $\hat{d} := L^T d \in \mathbb{R}^m$ and $\hat{R} := \frac{\tau}{\sigma}L^{-1}R \in \mathbb{R}^m$. Since $L^T$ is an upper triangular matrix, we can recover $d$ from $\hat{d}$ with the cost of $O(m^2)$. In the case of linear regression or the logistic regression problems, the cost of recovering $d$ from $\hat{d}$ is actually $O(m)$. Thus, we only need to focus on solving the linear system \eqref{eq: appendix_newton} for $\hat{d}$.

Based on the discussions in Proposition \ref{compute_M}, for each $j \in \{1,\cdots,l\}$, we can choose $M_j\in \partial_{\rm HS} {\rm Prox}_{\sigma\lambda\|({\cal P} w)^{(j)}\circ\cdot\|_1^2}(({\cal P} \hat{x})^{(j)})$ such that it has the following form:
\begin{align*}
M_j = {\rm Diag}(\xi_j)-\frac{2\sigma\lambda}{1+2\sigma\lambda (\tilde{w}_j^T\tilde{w}_j)} \tilde{w}_j\tilde{w}_j^T,
\end{align*}
where $\xi_j \in \mathbb{R}^{n_j-n_{j-1}}$ is a $0$-$1$ vector defined as $(\xi_j)_i = 0$ if $i\in I(|({\cal P} \hat{x})^{(j)}|)$, $(\xi_j)_i = 1$ otherwise, and $\tilde{w}_j=({\rm sign}(({\cal P} \hat{x})^{(j)})\circ \xi_j)\circ({\cal P} w)^{(j)}$, where $I(\cdot)$ is defined in \eqref{eq:Ia}.

We know that the costs of directly computing $\hat{A} {\cal M} \hat{A}^T$ and $\hat{A} {\cal M} \hat{A}^T\bar{d}$ for a given vector $\bar{d}\in\mathbb{R}^m$ are $O(m^2n)$ and $O(mn)$, respectively. This is computationally expensive when $m$ and $n$ are large. Next we will carefully explore the second-order sparsity of the underlying Jacobian which will substantially reduce the computational cost for solving the linear system \eqref{eq: appendix_newton}.

For each $j \in \{1,\cdots,l\}$, by taking advantage of the $0$-$1$ structure of $\xi_j$ and the definition of $\tilde{w}_j$, we have
\begin{align*}
M_j = {\rm Diag}(\xi_j)\left({\rm Diag}(\xi_j)-\frac{2\sigma\lambda}{1+2\sigma\lambda (\tilde{w}_j^T\tilde{w}_j)} \tilde{w}_j\tilde{w}_j^T\right){\rm Diag}(\xi_j)={\rm Diag}(\xi_j)M_j{\rm Diag}(\xi_j).
\end{align*}
Define $K_j:=\{k\mid (\xi_j)_k=1,k=1,\cdots,n_j-n_{j-1}\}$, $\xi:=[\xi_1;\cdots;\xi_l]\in \mathbb{R}^n$ and ${\cal K}:=\{k\mid (\xi)_k=1,k=1,\cdots,n\}$. It holds that
\begin{align}
\hat{A} {\cal M} \hat{A}^T = \hat{A}{\rm Diag}(\xi){\rm Diag}(M_1,\cdots,M_l){\rm Diag}(\xi) \hat{A}^T
=\hat{A}_{\cal K}{\rm Diag}(\hat{M}_1,\cdots,\hat{M}_l) \hat{A}_{\cal K}^T,\label{eq: appendix_AMAT}
\end{align}
where $\hat{A}_{\cal K}\in \mathbb{R}^{m\times |{\cal K}|}$ is the matrix consisting of the columns of $\hat{A}$ indexed by ${\cal K}$, and for each $j \in \{1,\cdots,l\}$, $\hat{M}_j\in \mathbb{R}^{| K_j| \times |K_j|}$ is defined as
\begin{align*}
\hat{M}_j = I_{|K_j|}-c_jv_jv_j^T,
\end{align*}
with $v_j:=(\tilde{w}_j)_{K_j}$, $c_j:=\frac{2\sigma\lambda}{1+2\sigma\lambda (\tilde{w}_j^T\tilde{w}_j)}$.

From the equation \eqref{eq: appendix_AMAT}, we can see that the costs of computing $\hat{A} {\cal M} \hat{A}^T$ and $\hat{A} {\cal M} \hat{A}^T \bar{d}$ for a given vector $\bar{d}\in \mathbb{R}^m$ reduce to $O(m^2|{\cal K}|)$ and $O(m|{\cal K}|)$, respectively. The reduction of the computation time is significant, since the sparsity of the solution induced by the exclusive lasso regularizer implies that $|{\cal K}|\ll n$. Note that when $m$ is moderate, we can use the Cholesky decomposition to solve the linear system \eqref{eq: appendix_newton} with the computational cost of $O(m^3+m^2|{\cal K}|)$. For the case when $|{\cal K}|\ll m$, we can use the Sherman-Morrison-Woodbury formula \citep{golub1996matrix} to further reduce the computational cost of solving \eqref{eq: appendix_newton}. To be specific, we have
\begin{align*}
\left(I_m+\hat{A}_{\cal K}{\rm Diag}(\hat{M}_1,\cdots,\hat{M}_l) \hat{A}_{\cal K}^T\right)^{-1}=I_m-\hat{A}_{\cal K}\left({\rm Diag}(\hat{M}_1^{-1},\cdots,\hat{M}_l^{-1})+\hat{A}_{\cal K}^T\hat{A}_{\cal K}
\right)^{-1}\hat{A}_{\cal K}^T,
\end{align*}
where for each $j \in \{1,\cdots,l\}$,
\begin{align*}
\hat{M}_j^{-1} = I_{|K_j|}+\left( c_j^{-1}-v_j^Tv_j\right)^{-1} v_jv_j^T.
\end{align*}
Now, the cost of solving \eqref{eq: appendix_newton} is reduced to $O(|{\cal K}|^3+|{\cal K}|^2m)$. For the case when $m$ and $|{\cal K}|$ are both large, we can employ the conjugate gradient (CG) method to solve \eqref{eq: appendix_newton}, where the computational cost of each iteration of CG method is $O(m|{\cal K}|)$.

As one can see, in our implementation, we fully take advantage of the sparsity of the solution and the structure of the underlying Jacobian to highly reduce the computational cost of solving the Newton system \eqref{eq: cg-system}, which makes our SSN method efficient and robust for large-scale problems.

\section{Numerical experiments}
\label{sec:numerical}
In this section, we perform numerical experiments to evaluate the performance of our proposed PPDNA for solving exclusive lasso models from two aspects:
\begin{itemize}[noitemsep,topsep=0pt]
	\item[(1)] We compare our proposed PPDNA for solving exclusive lasso models with other four popular algorithms. The numerical results show that the PPDNA outperms other algorithms for solving the exclusive lasso model by a large margin.
	\item[(2)] We apply the exclusive lasso models to some real application problems, including the index ETF in finance, and image and text classifications.
\end{itemize}
All our computational results are obtained by running {\sc Matlab} on a windows workstation (Intel(R) Core(TM) i7-8700 CPU @ 3.20GHz, 64G RAM).

We terminate the tested algorithms when $\eta_{\rm KKT}\leq \varepsilon$, where $\varepsilon > 0$ is a given tolerance, which is set to be $10^{-6}$ by default. To measure the accuracy of the obtained solution by each algorithm, we use the following relative KKT residual:
\begin{align*}
	\eta_{\rm KKT} := \frac{\|x - {\rm Prox}_{\lambda p}(x - A^T\nabla h(Ax)\|}{1 + \|x\| + \|A^T\nabla h(Ax)\|}.
\end{align*}

\subsection{Performance of the PPDNA for solving exclusive lasso problems}
\label{sec: numerical PPDNA}
In this subsection, we compare the proposed PPDNA for solving exclusive lasso models with a given $\lambda>0$ to four popular first-order algorithms: ILSA  \citep{kong2014exclusive}, ADMM with the step length $\kappa = 1.618$ \citep{fazel2013hankel}, APG with restart under the setting described in \citep{becker2011templates} and the coordinate descent algorithm in the R package `ExclusiveLasso'\footnote{https://github.com/DataSlingers/ExclusiveLasso}, which is widely used in the statistics community. In the experiments, we also terminate PPDNA when it reaches the maximum iteration  of $200$, and terminate ILSA, ADMM and APG when they reach the maximum iteration  of $200,000$. We reset the maximum iteration of the algorithm in the R package `ExclusiveLasso' to $50,000$ to try to obtain a solution with a relatively high accuracy, and keep all the other settings as default in the solver. Here, we also add one command in this R solver to let it return the solution of the last iteration if the maximum iteration has been reached. In addition, we set the maximum computation time of each experiment as one hour. To demonstrate the efficiency and scalability of the algorithms, we perform the time comparison on synthetic datasets over a range of scales.

For simplicity, we first take the weight vector $w$ to be all ones, and the vector $c$ to be zero. The exclusive lasso model can be described as
\begin{equation}\label{eq: exclusive-lasso}
	\min_{x\in \mathbb{R}^n} \ \Big\{h(Ax) + \lambda \sum_{j=1}^l \| x_{g_j}\|_1^2\Big\}.
\end{equation}

\subsubsection{The regularized linear regression problem with synthetic data}\label{sec: exper_linear}
 In the model \eqref{eq: exclusive-lasso}, we take $h(y):=\sum_{i=1}^m (y_i-b_i)^2/2$, where $b\in \mathbb{R}^m$ is given. Motivated by \citep{campbell2017within}, we generate the synthetic data using the model $b = Ax ^* +\epsilon$, where $x ^*$ is the predefined true solution and $\epsilon \sim \mathcal{N}(0, I_m)$ is a random noise vector. Given the number of observations $m$, the number of groups $l$ and the number of features $p$ in each group, we generate each row of the matrix $A \in \mathbb{R}^{m \times lp}$ by independently sampling a vector from a multivariate normal distribution $\mathcal{N}(0, \Sigma)$, where $\Sigma$ is a Toeplitz covariance matrix with entries $\Sigma_{ij} = 0.9^{|i - j|}$ for features in the same group, and $\Sigma_{ij} = 0.3^{|i - j|}$ for features in different groups. For the ground-truth $x^{*}$, we randomly generate $10$ nonzero elements in each group with i.i.d values drawn from the uniform distribution on $[0,10]$.

We mainly focus on solving the exclusive lasso model in the high-dimensional settings. Hence, we fix $m$ to be $200$ and $l$ to be $20$, but vary the number of features $p$ in each group from $50$ to $1000$. That is, we vary the total number of features $n=lp$ from $1000$ to $20000$. To compare the robustness of different algorithms with respect to the hyper-parameter $\lambda$, we test all the algorithms under three different values of $\lambda$. The time comparison for $\varepsilon = 10^{-6}$ is shown in Figure \ref{fig: time-comparison_2}, which demonstrates the superior performance of the PPDNA, especially for large-scale instances, comparing to ILSA, ADMM, APG and CD. As one can observe, for the largest instance with $\lambda = 0.1$ or $0.001$, PPDNA is at least one hundred times faster than ADMM, which is the best performing first-order method. We find that the CD scheme implemented in the R solver can not solve the exclusive lasso problems with a small $\lambda$ to a high accuracy within the given maximum number of iterations. For a better illustration, we also report the time comparison of the five algorithms for a moderate accuracy with $\varepsilon = 10^{-4}$ in Figure \ref{fig: time-comparison}. From Figure \ref{fig: time-comparison}, we can see that for large $\lambda=10$, PPDNA performs the best, ADMM, APG and CD also give satisfactory performance, while ILSA gives the worst performance. For median $\lambda=0.1$, PPDNA takes seconds to solve each problem, while, ADMM,  the best performing algorithm in the remaining methods,   needs about tens of seconds. For small $\lambda=0.001$, PPDNA still only needs seconds to solve each instance. However, ILSA and ADMM each takes up to 100 seconds, and APG and CD are even unable to solve many instances. These experiments show that PPDNA outperforms the existing solvers for the exclusive lasso problems under both moderate and high accuracy scenarios.

\begin{figure}[H]
	\begin{center}
		\includegraphics[width = 1\columnwidth]{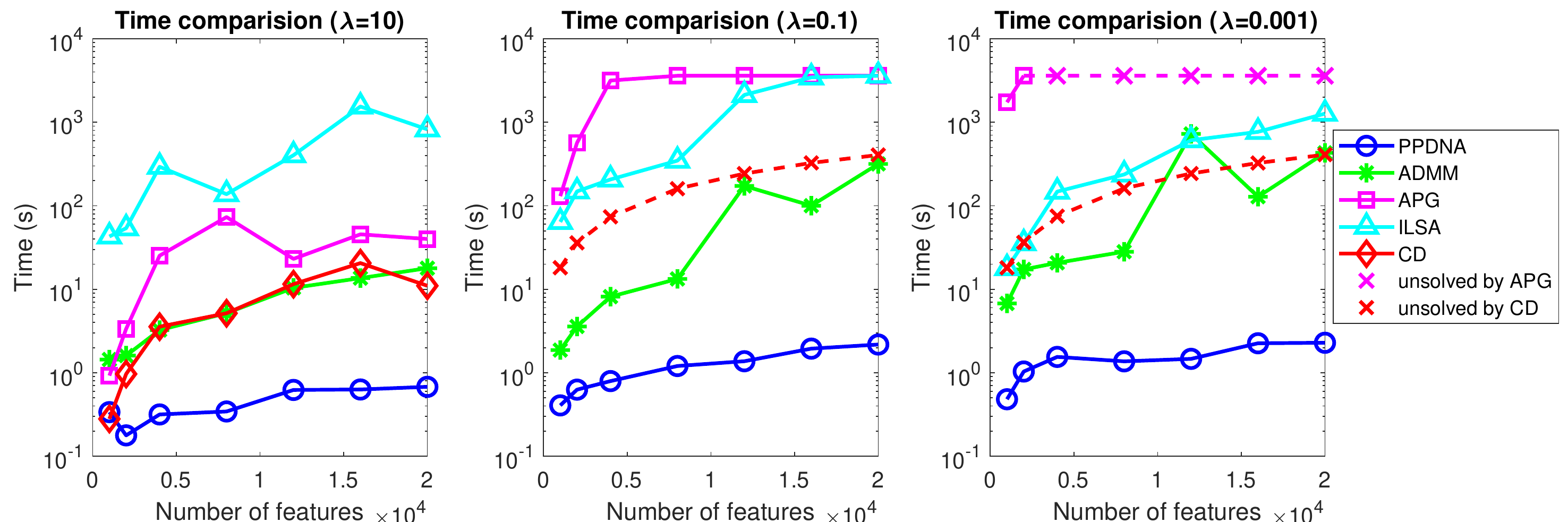}
		\setlength{\abovecaptionskip}{0pt}
		\setlength{\belowcaptionskip}{-15pt}
		\caption{Time comparison among PPDNA, ILSA, ADMM, APG and CD for linear regression with an exclusive lasso regularizer on synthetic datasets (with stopping criterion $\eta_{\rm KKT}\leq 10^{-6}$). }
		\label{fig: time-comparison_2}
	\end{center}
\end{figure}

\vspace{-0.75cm}

\begin{figure}[H]
	\begin{center}
		\includegraphics[width = 1\columnwidth]{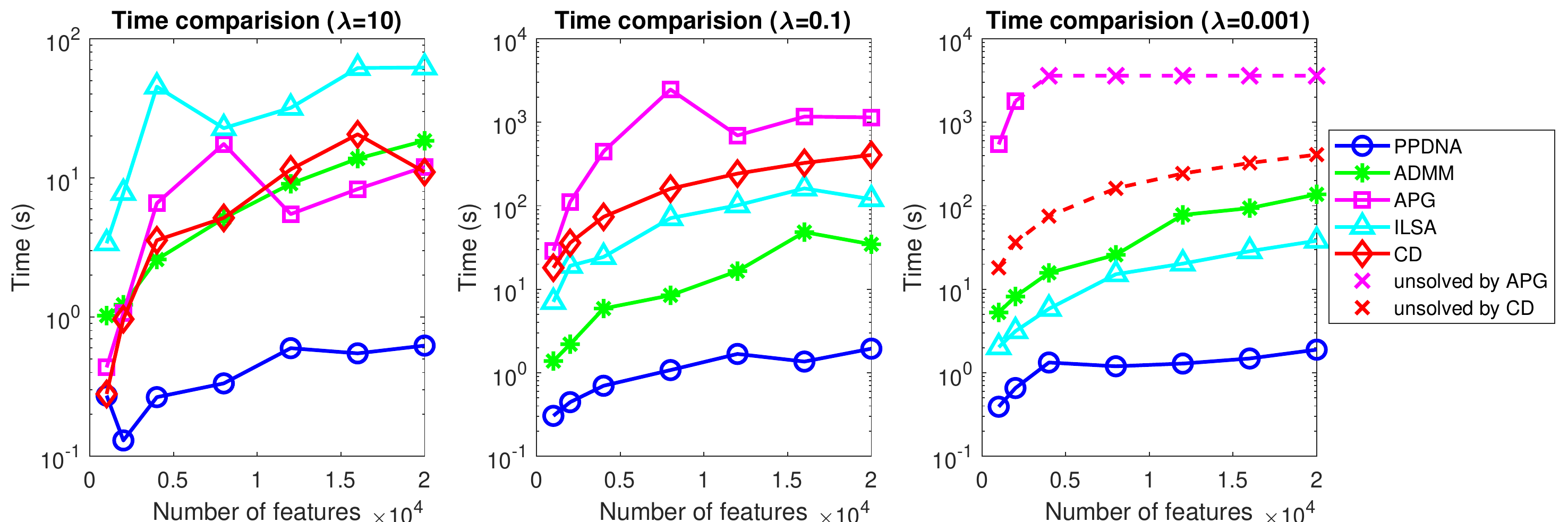}
		\setlength{\abovecaptionskip}{0pt}
		\setlength{\belowcaptionskip}{-15pt}
		\caption{Same setting as Figure \ref{fig: time-comparison_2} with stopping criterion $\eta_{\rm KKT}\leq 10^{-4}$).}
		\label{fig: time-comparison}
	\end{center}
\end{figure}

More numerical results on higher-dimensional cases (larger $m$, larger $l$ and/or larger $p$) with the stopping criterion $\eta_{\rm KKT}\leq 10^{-6}$ are shown in Table \ref{tab: ls}. For testing purposes, the regularization parameter $\lambda$ in the problem \eqref{eq: exclusive-lasso} is chosen as $\lambda = \lambda_b \|A^T b\|_{\infty}$, where $0<\lambda_b<1$. As one can see from Figures \ref{fig: time-comparison_2} and \ref{fig: time-comparison}, APG, ILSA and CD are not efficient enough to solve large-scale instances. Thus we only compare PPDNA with ADMM in these higher-dimensional cases. Here, we set the maximum iteration number of ADMM as $500,000$. For the instances in Table \ref{tab: ls}, PPDNA is able to solve most of the problems within one minute, whereas ADMM takes dozens of times longer.

\begin{table}[H]
	\tbl{Comparison between PPDNA and ADMM for linear regression with an exclusive lasso regularizer on synthetic datasets. In the table, `13(48)' means `PPA iterations (total inner SSN iterations)'. Time is in the format of (hours:minutes:seconds). $\sigma_{\max}(\cdot)$ means the largest singular value and ${\rm cond}(\cdot)$ denotes the condition number of a matrix, which is the ratio of the largest singular value to the smallest singular value.}
	{\begin{tabular}{ccccc} \toprule
			& & iter &  $\eta_{\rm KKT}$ & time\\ \cmidrule{3-5}
			Data $(m,l,p)$  & $\lambda_b$  & PPDNA $|$ ADMM & PPDNA $|$ ADMM & PPDNA $|$ ADMM \\ \midrule
			\multirow{2}*{\tabincell{c}{$(500,20,2000)$ \\
			$\sigma_{\max}(A) = 272, {\rm cond}(A) = 1.9582$ } }
			&1e-3 & 13(48) $|$ 3278  & 4.7e-7 $|$ 9.9e-7 & 0:00:02 $|$ 0:01:05  \\
			&1e-5 & 23(103) $|$ 14459  & 4.2e-7 $|$ 1.0e-6  & 0:00:05 $|$ 0:06:11  \\
			\midrule
			\multirow{2}*{\tabincell{c}{$(500,20,3000)$ \\
			$\sigma_{\max}(A) = 317, {\rm cond}(A) = 1.7395$ } }
			&1e-3 & 13(47) $|$ 3777  & 9.2e-7 $|$ 1.0e-6 & 0:00:03 $|$ 0:01:52  \\
			&1e-5 & 23(101) $|$ 45018  & 3.9e-7 $|$ 1.0e-6 & 0:00:06 $|$ 0:28:25 \\
			\midrule
			\multirow{2}*{\tabincell{c}{$(1000,20,2000)$ \\
			$\sigma_{\max}(A) = 304, {\rm cond}(A) = 2.5902$ } }
			&1e-3 & 11(50) $|$ 3302  & 1.4e-7 $|$ 9.9e-7 & 0:00:04 $|$ 0:01:57  \\
			&1e-5 & 21(120) $|$ 12986  & 1.9e-7 $|$ 1.0e-6 & 0:00:12 $|$ 0:09:47 \\
			\midrule
			\multirow{2}*{\tabincell{c}{$(1000,20,4000)$ \\
			$\sigma_{\max}(A) = 386, {\rm cond}(A) = 1.9742$ } }
			&1e-3 & 11(45) $|$ 3254  & 8.5e-7 $|$ 9.9e-7 & 0:00:05 $|$ 0:03:59  \\
			&1e-5 & 22(123) $|$ 27820  & 3.0e-7 $|$ 1.0e-6 & 0:00:17 $|$ 0:40:43 \\
			\midrule
			\multirow{2}*{\tabincell{c}{$(5000,20,1000)$ \\
			$\sigma_{\max}(A) = 401, {\rm cond}(A) = 14.1741$ } }
			&1e-3 & 7(43) $|$ 3576  & 2.7e-7 $|$ 9.8e-7 & 0:00:04 $|$ 0:06:25  \\
			&1e-5 & 12(165) $|$ 3456  & 1.3e-7 $|$ 9.9e-7 & 0:00:50 $|$ 0:06:47 \\
			\midrule
			\multirow{2}*{\tabincell{c}{$(5000,50,1000)$ \\
			$\sigma_{\max}(A) = 474, {\rm cond}(A) = 6.158$ } }
			&1e-3 & 8(54) $|$ 3664  & 8.4e-8 $|$ 9.9e-7 & 0:00:14 $|$ 0:14:25  \\
			&1e-5 & 14(181) $|$ 3562  & 2.2e-7 $|$ 9.9e-7 & 0:02:25 $|$ 0:14:50 \\
			\bottomrule
	\end{tabular}}
	\label{tab: ls}
\end{table}

\subsubsection{The regularized logistic regression problem with synthetic data}  To test the regularized logistic regression problem, we take $h(y) = \sum_{i=1}^m \log(1 + \exp(-b_i y_i))$ in \eqref{eq: exclusive-lasso}, where $b\in \{-1,1\}^m$ is given. We use the same synthetic datasets described in the previous part, except for letting $b_i=1$ if $ Ax ^* +\varepsilon\geq 0$, and $-1$ otherwise, where $\varepsilon \sim \mathcal{N}(0, I_m)$ is a random noise vector. As one can see in the previous experiments, APG, ILSA and CD are very time-consuming when solving large-scale exclusive lasso problems compared to PPDNA and ADMM. Thus for logistic regression problems, we only compare PPDNA with ADMM. The numerical results are shown in Table \ref{tab: logistic}, where the regularization parameter $\lambda$ in the exclusive lasso problem \eqref{eq: exclusive-lasso} is
chosen as $\lambda = \lambda_b \|A^T b\|_{\infty}$. Again, we can observe the superior performance of PPDNA against ADMM, and the performance gap is especially wide when the parameter $\lambda_b=10^{-5}$. For example, PPDNA is at least $160$ times faster than ADMM in solving the instance $(500,20,5000)$ with $\lambda_b=10^{-5}$.

\begin{table}[H]
	\tbl{Time comparison between PPDNA and ADMM for logistic regression with an exclusive lasso regularizer on synthetic datasets. A value in bold means that the algorithm fails to solve the instance to the required accuracy.}
	{\begin{tabular}{ccccc} \toprule
			& & iter &  $\eta_{\rm KKT}$ & time\\ \cmidrule{3-5}
			Data $(m,l,p)$  & $\lambda_b$  & PPDNA $|$ ADMM & PPDNA $|$ ADMM & PPDNA $|$ ADMM \\ \midrule
			\multirow{2}*{\tabincell{c}{$(500,20,5000)$\\
		    $\sigma_{\max}(A) = 386, {\rm cond}(A) = 1.5352$  } }
			&1e-3 & 16(43) $|$ 2252  & 3.6e-7 $|$ 1.0e-6 & 0:00:05 $|$ 0:08:14 \\
			&1e-5 & 46(54) $|$ 6249  & 8.3e-7 $|$ 1.0e-6 & 0:00:06 $|$ 0:16:31 \\
			\midrule
			\multirow{2}*{\tabincell{c}{$(1000,20,8000)$ \\
			$\sigma_{\max}(A) = 503, {\rm cond}(A) = 1.6212$  } }
			&1e-3 & 12(47) $|$ 1478  & 2.5e-7 $|$ 1.0e-6 & 0:00:12 $|$ 0:18:20 \\
			&1e-5 & 67(75) $|$ 6284  & 9.2e-7 $|$ 1.0e-6 & 0:00:21 $|$ 0:57:17 \\
			\midrule
			\multirow{2}*{\tabincell{c}{$(2000,20,10000)$ \\
			$\sigma_{\max}(A) = 595, {\rm cond}(A) = 1.8496$  } }
			&1e-3 & 10(54) $|$ 1685  & 1.9e-7 $|$ 9.9e-7 & 0:00:31 $|$ 0:57:59 \\
			&1e-5 & 45(64) $|$ 2255  & 5.5e-7 $|$ \bf{1.3e-4} & 0:00:47 $|$ 1:00:02 \\
			\midrule
			\multirow{2}*{\tabincell{c}{$(5000,20,1000)$ \\
			$\sigma_{\max}(A) = 401, {\rm cond}(A) = 14.1741$  } }
			&1e-3 & 9(85) $|$ 1178  & 2.6e-7 $|$ 1.0e-6 & 0:01:21 $|$ 0:52:44 \\
			&1e-5 & 13(84) $|$ 1523  & 3.2e-7 $|$ 1.0e-6 & 0:02:01 $|$ 0:48:00 \\
			\midrule
			\multirow{2}*{\tabincell{c}{$(5000,50,5000)$ \\
			$\sigma_{\max}(A) = 474, {\rm cond}(A) = 6.158$  } }
			&1e-3 & 9(68) $|$ 943  & 8.3e-8 $|$ 1.0e-6 & 0:01:52 $|$ 0:57:33 \\
			&1e-5 & 15(63) $|$ 1350  & 1.4e-7 $|$ \bf{2.4e-4} & 0:02:37 $|$ 1:00:01 \\
			\bottomrule
	\end{tabular}}
	\label{tab: logistic}
\end{table}

\subsubsection{Regularized exclusive lasso problems with non-uniform weights}
In order to better assess the robustness and efficiency of our proposed algorithm, we now move on to the case where the weight vector of the exclusive lasso problem is non-uniform. Specifically, we consider the weighted exclusive lasso regularized linear regression problem
\begin{equation}\label{eq: w-exclusive-lasso}
	\min_{x\in \mathbb{R}^n} \ \Big\{h(Ax) + \lambda \sum_{j=1}^l \| w_{g_j} \circ x_{g_j}\|_1^2\Big\},
\end{equation}
where $h(y):=\sum_{i=1}^m (y_i-b_i)^2/2$ and $w\in\mathbb{R}^n$ is a given weight vector. In the experiments, we generate each element of the weight vector uniformly random on $[0,1]$ and then follow the same procedure in Section \ref{sec: exper_linear} to generate the remaining data.

Since ADMM outperforms the algorithms APG, ILSA and CD for solving large-scale uniformly-weighted exclusive lasso problem, we focus on comparing PPDNA and ADMM on solving the problem \eqref{eq: w-exclusive-lasso}. Note that compared to the experiments in Section \ref{sec: exper_linear}, we pick two different choices of the parameter $\lambda_b$ to test the robustness of the algorithms as well as to get reasonable number of non-zero elements in the obtained solutions. Detailed numerical results are shown in Table \ref{tab: ls_weight}. We can see that the overall performance of the two algorithms on the uniformly weighted cases and non-uniformly weighted cases are quite similar. Specifically, PPDNA can solve most of the instances within twenty seconds, while ADMM takes much longer computational time.

\begin{table}[H]
	\tbl{Comparison between PPDNA and ADMM for linear regression with a weighted exclusive lasso regularizer on synthetic datasets.}
	{\begin{tabular}{ccccc} \toprule
		    & & iter &  $\eta_{\rm KKT}$ & time\\ \cmidrule{3-5}
			Data $(m,l,p)$  & $\lambda_b$  & PPDNA $|$ ADMM & PPDNA $|$ ADMM & PPDNA $|$ ADMM \\ \midrule
			\multirow{2}*{\tabincell{c}{$(500,20,2000)$  } }
			&1e-1 & 18(78) $|$ 7867  & 7.7e-7 $|$ 1.0e-6 & 0:00:04 $|$ 0:03:01  \\
			&1e-3 & 27(105) $|$ 27001  & 5.9e-7 $|$ 1.0e-6  & 0:00:06 $|$ 0:11:21  \\
			\midrule
			\multirow{2}*{\tabincell{c}{$(500,20,3000)$  } }
			&1e-1 & 21(92) $|$ 10995  & 6.2e-7 $|$ 1.0e-6 & 0:00:07 $|$ 0:06:39  \\
			&1e-3 & 29(108) $|$ 65933  & 2.2e-7 $|$ 1.0e-6 & 0:00:09 $|$ 0:43:05 \\
			\midrule
			\multirow{2}*{\tabincell{c}{$(1000,20,2000)$  } }
			&1e-1 & 15(64) $|$ 3358  & 2.5e-7 $|$ 1.0e-6 & 0:00:06 $|$ 0:01:46  \\
			&1e-3 & 25(120) $|$ 33802  & 4.7e-7 $|$ 1.0e-6 & 0:00:13 $|$ 0:22:50 \\
			\midrule
			\multirow{2}*{\tabincell{c}{$(1000,20,4000)$  } }
			&1e-1 & 19(93) $|$ 9278  & 2.2e-7 $|$ 1.0e-6 & 0:00:14 $|$ 0:11:06  \\
			&1e-3 & 26(125) $|$ 47640  & 8.1e-7 $|$ \bf{4.8e-6} & 0:00:20 $|$ 1:00:00 \\
			\midrule
			\multirow{2}*{\tabincell{c}{$(5000,20,1000)$  } }
			&1e-1 & 7(37) $|$ 3488  & 4.4e-8 $|$ 9.9e-7 & 0:00:05 $|$ 0:05:00  \\
			&1e-3 & 9(45) $|$ 3205  & 3.6e-7 $|$ 1.0e-6 & 0:00:16 $|$ 0:04:41 \\
			\midrule
			\multirow{2}*{\tabincell{c}{$(5000,50,1000)$  } }
			&1e-1 & 7(45) $|$ 3258  & 7.9e-7 $|$ 9.8e-7 & 0:00:18 $|$ 0:09:31  \\
			&1e-3 & 17(100) $|$ 3266  & 3.8e-8 $|$ 1.0e-6 & 0:03:26 $|$ 0:09:25 \\
			\bottomrule
	\end{tabular}}
	\label{tab: ls_weight}
\end{table}

\subsection{Real applications}\label{sec: real}
In this subsection, we apply the exclusive lasso model to some real application problems, including the index exchange traded fund (ETF) in finance, image and text classifications in multi-class classifications.

\subsubsection{Index exchange-traded fund}
Consider the portfolio selection problem where a fund manager wants to select a small subset of stocks to track the S\&P 500 index. In order to diversify the risks, the portfolio is required to span across all sectors. Such an application naturally leads us to consider the exclusive lasso model.

In our experiments, we download all the stock price data in the US market between 2018-01-01 and 2018-12-31 (251 trading days) from Yahoo finance\footnote{https://finance.yahoo.com}, and drop the stocks with more than 10\% of their price data being missed. We get 3074 stocks in our stock universe and handle the missing data via the common practice of forward interpolation. Then we denote the historical daily return matrix as $R \in \mathbb{R}^{250 \times 3074}$, and the daily return of the S\&P 500 index as $y \in \mathbb{R}^{250}$. Since there are 12 sectors in the US market (e.g., finance, healthcare), we have a natural group partition for our stock universe as $\mathcal{G} = \{g_1, g_2, \dots, g_{12}\}$, where $g_i$ is the index set for stocks in the $i$-th sector.

To test the performance of the exclusive lasso model in index tracking, we use the rolling window method \citep[Chapter 9]{zivot2006modeling} to test the in-sample and out-of-sample performance of the model. We use the historical data in the last 90 trading days to estimate a portfolio vector via the model for the future 10 days. In each experiment, we scale the feature matrix $A$ and the response vector $b$ by $1/\sqrt{\|A\|_F}$, and select the parameter $\lambda$ in the range of $10^{-3}$ to $10^{-5}$ with $20$ equally divided grid points on the $\log_{10}$ scale, using 9-fold cross-validation. The in-sample and out-of-sample performance of the exclusive lasso model, the lasso model and the group lasso model is shown in Figure \ref{fig: partial-index-tracking}. The out-of-sample performance of the exclusive lasso model is visibly better than those corresponding to the lasso and group lasso models.

\begin{figure}[H]
	\begin{center}
		\includegraphics[width = 0.4\columnwidth]{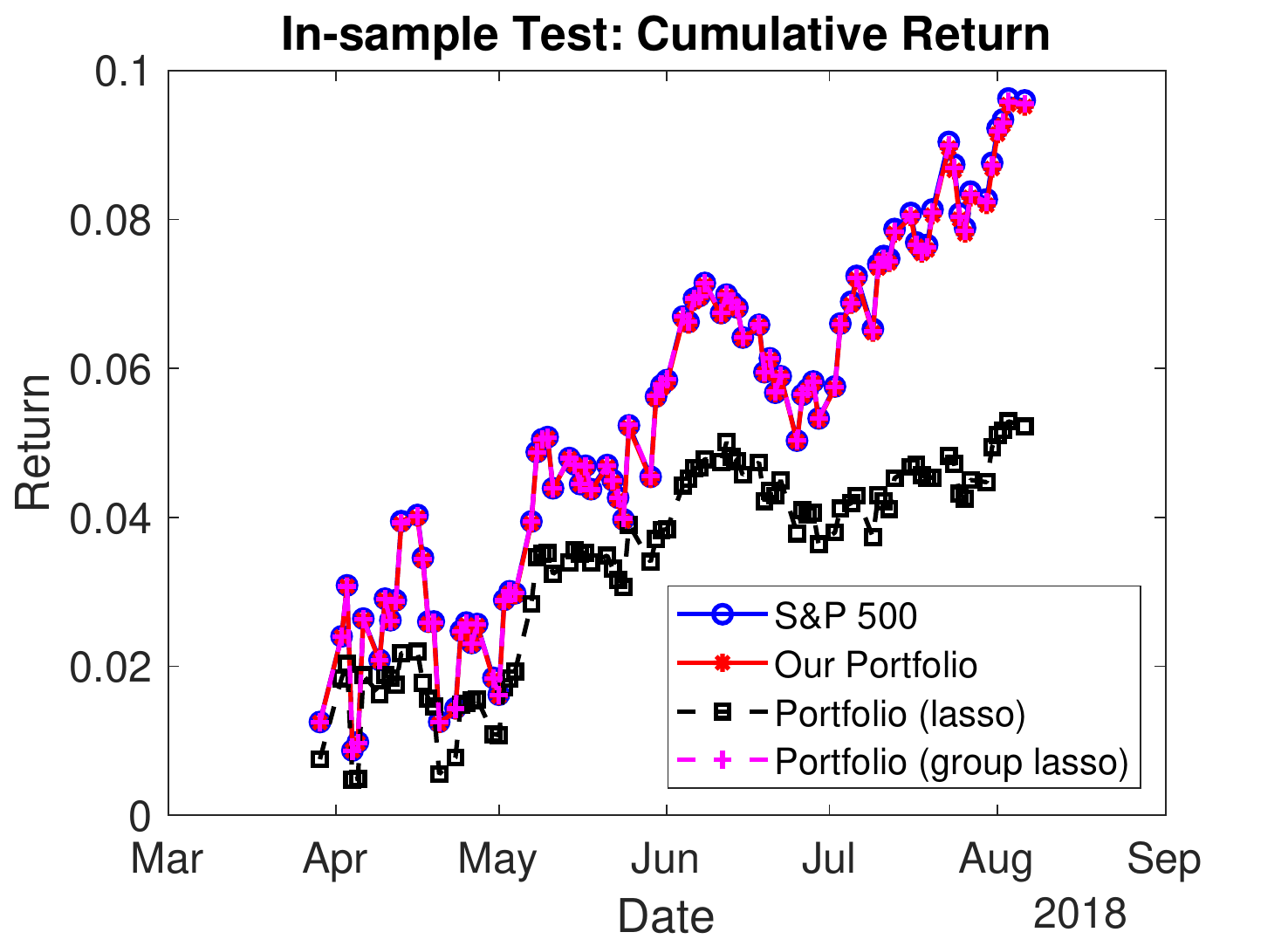}
		\quad
		\includegraphics[width = 0.4\columnwidth]{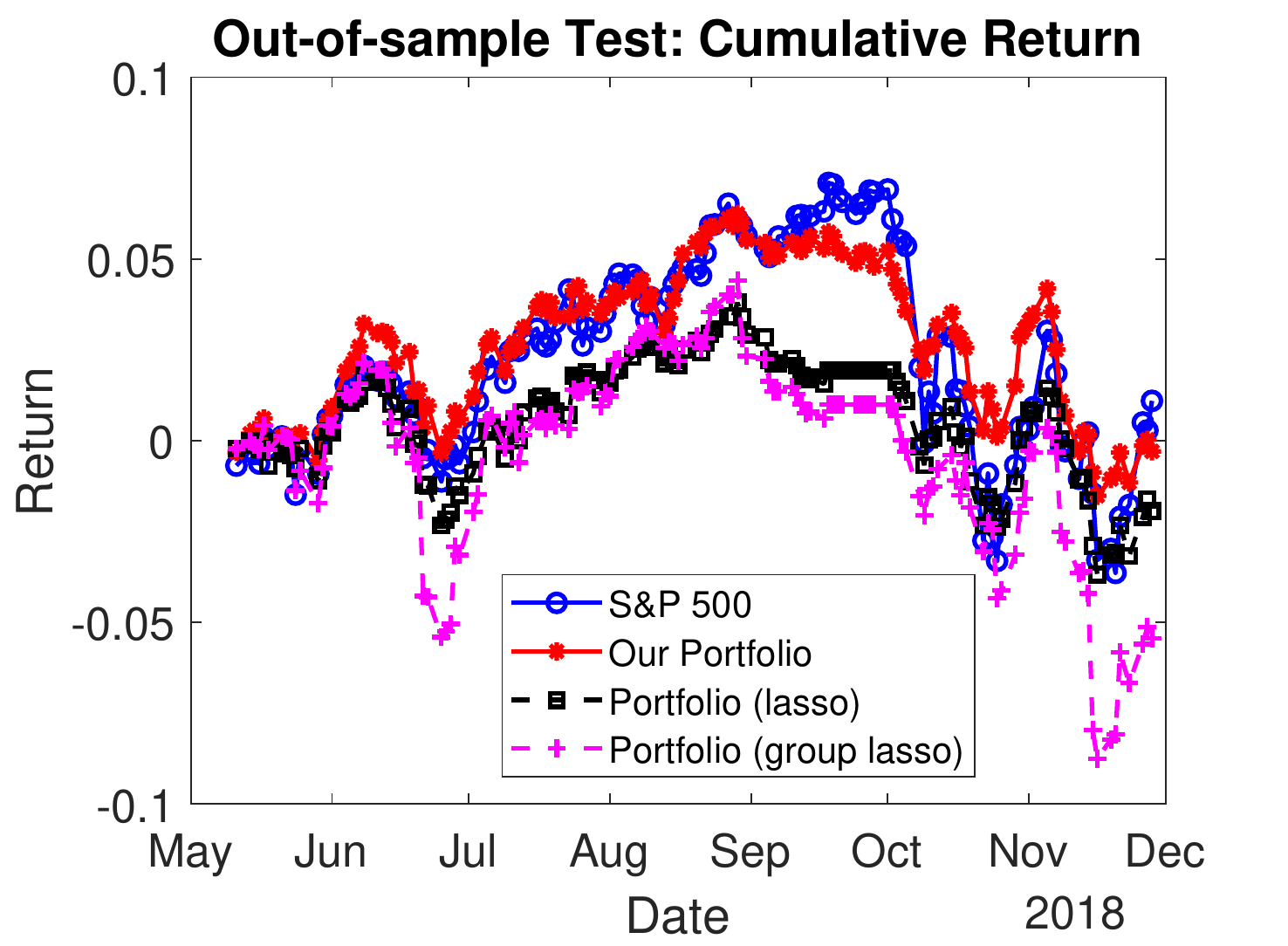}
		\setlength{\abovecaptionskip}{-2pt}
		\setlength{\belowcaptionskip}{-15pt}
		\caption{In-sample and out-of-sample performance of the exclusive lasso, the group lasso and the lasso model for index tracking of S\&P 500.}
		\label{fig: partial-index-tracking}
	\end{center}
\end{figure}

We plot the percentage of stocks from each sector in the portfolio obtained from the three tested models in Figure \ref{fig: pie-fig}. The result shows that the exclusive lasso model can select stocks from all the 12 sectors, but the lasso model selects stocks only from 10 sectors and the group lasso model selects stocks only from 7 sectors in the universe.

\begin{figure}[h]
	\vspace{-0.5cm}
	\hspace{-1cm}
	\subfloat[the exclusive lasso model]{
		\label{fig_eg1}
		\includegraphics[width = 0.55\columnwidth]{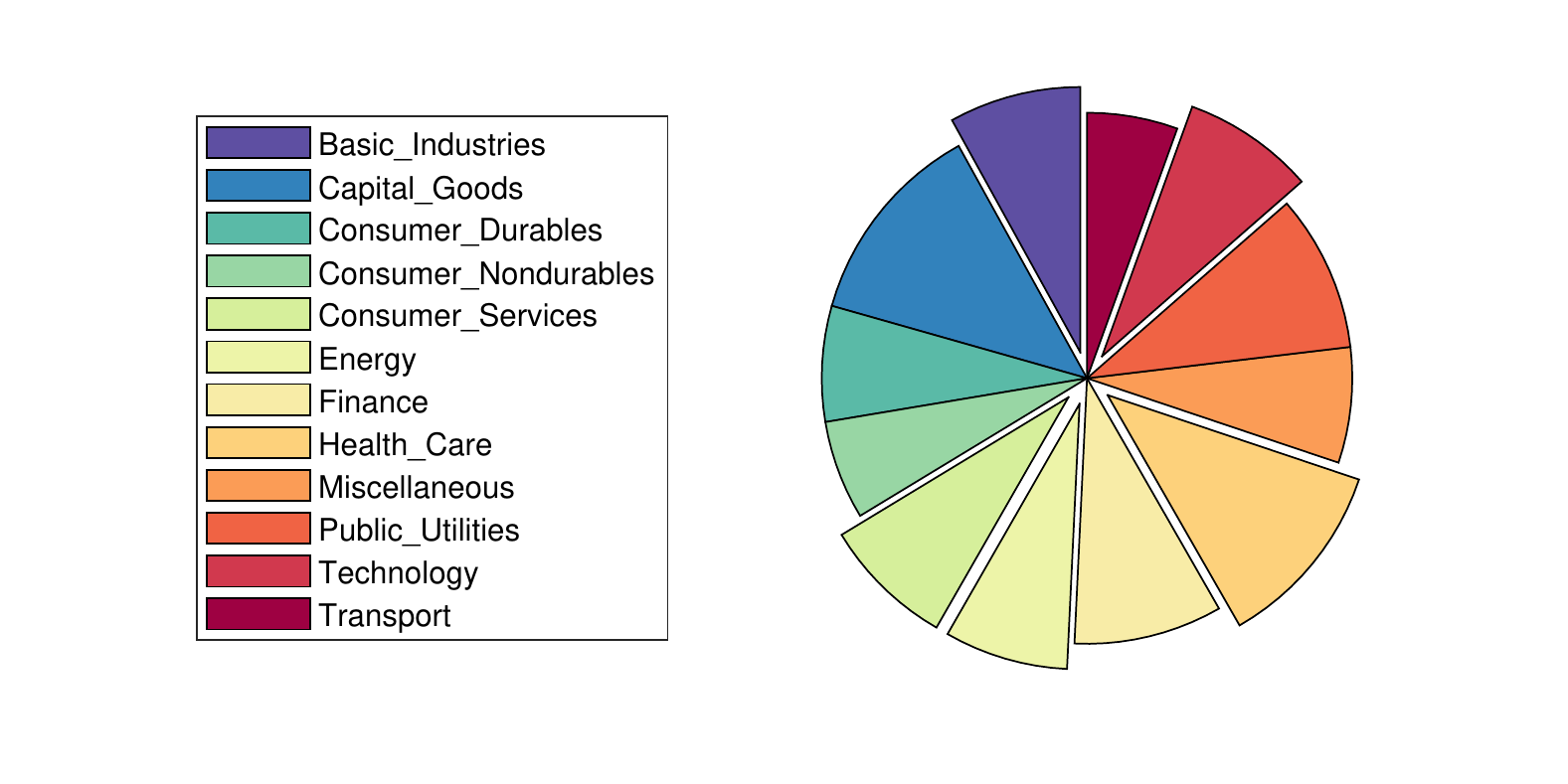}}
	\hspace{-1.2cm}
	\subfloat[the lasso model]{
		\label{fig_eg2}
		\includegraphics[width = 0.28\columnwidth]{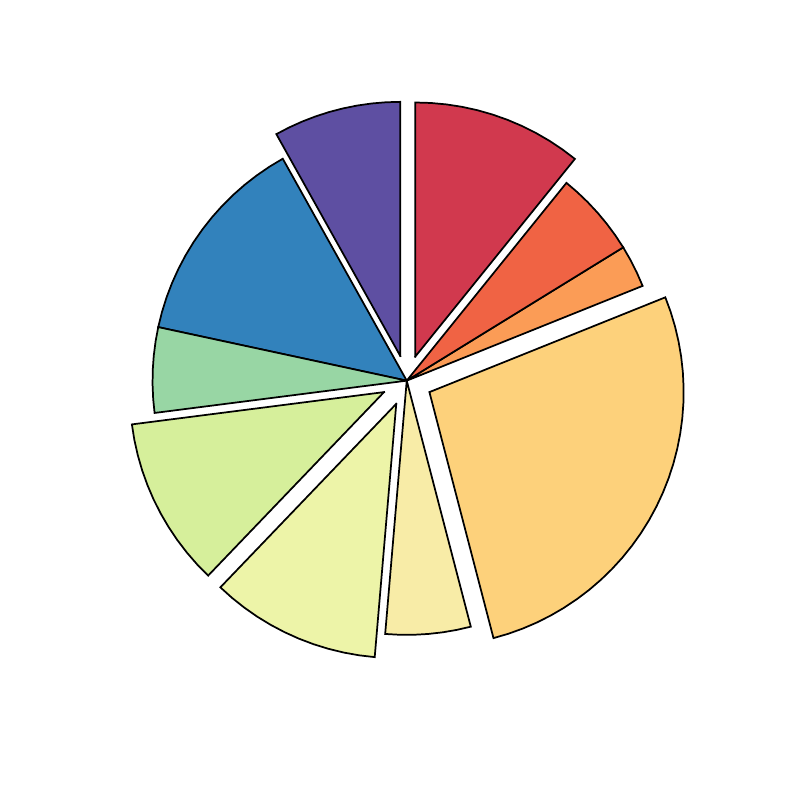}}
	\hspace{-0.8cm}
	\subfloat[the group lasso model]{
		\label{fig_eg3}
		\includegraphics[width = 0.28\columnwidth]{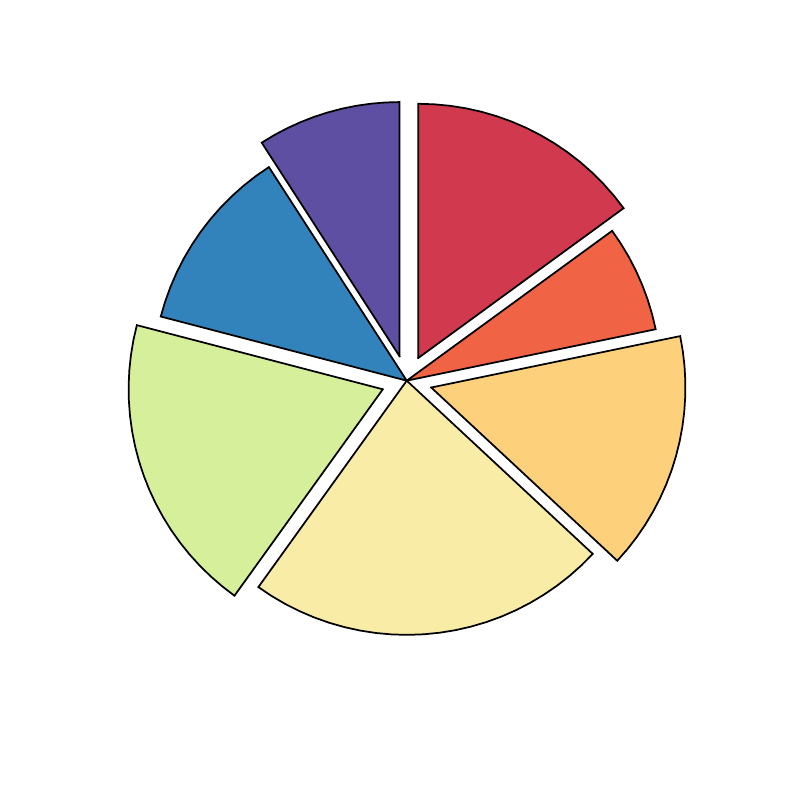}}
	\caption{Percentage of selected stocks by sectors.}
	\label{fig: pie-fig}
\end{figure}

\subsubsection{Image and text classifications}
We test the exclusive lasso model on multi-class classifications. For a given $k$-class classification dataset $\{(a_i, b_i)\}_{i=1}^N$, where $a_i \in \mathbb{R}^{p}$ is the feature vector and $b_i \in \mathbb{R}^k$ is the one-hot representation of the label, the exclusive lasso regression model for this problem \citep{zhou2010exclusive,kong2014exclusive,campbell2017within} is given by:
\begin{align}
	\min_{X\in \mathbb{R}^{p\times k}}\ \Big\{\frac{1}{2}\|AX-b\|_F^2+\lambda \sum_{j=1}^p\|X_{j,:}\|_1^2
	\Big\},\tag{M1}\label{eq: classification_M1}
\end{align}
where $A = [a_1, a_2, \dots, a_N]^T \in \mathbb{R}^{N \times p}$ and $b = [b_1, b_2, \dots, b_N]^T \in \mathbb{R}^{N \times k}$. The key motivation for considering this model is  to capture the negative correlation among the classes. However, the exclusive lasso regularizer may not exclude uninformative features if we penalize $X$ row-wise since it prefers to select at least one representative from each feature group. This phenomenon has also been discussed in a recent paper \citep{ming2019robust}. In our experiments, we consider the following model instead:
\begin{align}
	\min_{X\in \mathbb{R}^{p\times k}}\ \Big\{\frac{1}{2}\|AX-b\|_F^2+\lambda \sum_{j=1}^k\|X_{:,j}\|_1^2
	\Big\}.\tag{M2}\label{eq: classification_M2}
\end{align}
The motivation for considering model \eqref{eq: classification_M2} is that we can do class-wise feature selections, since the informative features for different classes are usually not identical. Also, uninformative features will  automatically be excluded by the nature of class-wise feature selections. In order to show that the new model we suggest is meaningful, we first compare the model performance on two popular real datasets: MNIST \citep{lecun1998gradient} and 20 Newsgroups\footnote{http://qwone.com/$\sim$jason/20Newsgroups/}. We summarize the details of the datasets in Table \ref{tab:real_datasets}. Note that, after vectorization, the target problem size is actually $kN \times kp$.

\begin{table}[H]
	\tbl{Details of real datasets.}
	{\begin{tabular}{lcccccc} \toprule
			Dataset & Num. of  classes $k$ & Num. of samples $N$ & Num. of features $p$ &  \multicolumn{3}{c}{Target problem}   \\
			 & & &  & $(m,n)=(kN,kp)$ & $\sigma_{\max}(A)$ & ${\rm cond}(A)$\\ \midrule
			MNIST & 10 & 60000 & 784 & (600000, 7840) & 1.5e3 & Inf\\
			\midrule
			20 Newsgroups & 20 & 11314 & 26214 & (226280, 524280)& 4.7 & Inf \\
			\bottomrule
	\end{tabular}}
	\label{tab:real_datasets}
\end{table}

\begin{table}[H]
	\tbl{Model comparison on real datasets.}
	{\begin{tabular}{lcccccc} \toprule
			Dataset & Model & $\lambda^*$  & total selected unique features & ${\rm nnz}(X)$ & training accuracy  & testing accuracy \\
			\midrule
			\multirow{2}*{\tabincell{c}{ MNIST } }
			& (M1) & 1.0e-1 & 717 & 1922 & 84.01\% & 84.64\%  \\
			& (M2) & 1.0e-3 & 449 & 1818 & 84.03\% & 84.79\%  \\
			\midrule
			\multirow{2}*{\tabincell{c}{ 20 Newsgroups } }
			& (M1) & 1.0e-3 & 25714 & 27537 & 88.15\% & 77.46\%  \\
			& (M2) & 1.0e-6 & 2789 & 5942 & 91.70\% & 79.14\%  \\
			\bottomrule
	\end{tabular}}
	\label{tab: numerical_results_realdata}
\end{table}

We train \eqref{eq: classification_M1} and \eqref{eq: classification_M2} independently on the two datasets. As prior knowledge, a certain percentage of features are uninformative for these datasets (e.g., background pixels for the MNIST dataset and some uninformative words for the 20 Newsgroups dataset). Thus in each experiment, we set a lower bound for the value of $\lambda$ such that no more than 90\% features are selected by the model. As a result, for the MNIST dataset, we train \eqref{eq: classification_M1} with $\lambda$ in the range from $10$ to $0.1$ and \eqref{eq: classification_M2} with $\lambda$ from $10$ to $10^{-3}$ with grid search and cross-validation. Similarly, for the 20 Newsgroup dataset, we train \eqref{eq: classification_M1} with $\lambda$ from $1$ to $10^{-3}$ and \eqref{eq: classification_M2} with $\lambda$ from $1$ to $10^{-6}$. We summarize the results in Table \ref{tab: numerical_results_realdata} and Figure \ref{fig: real_model_compare}. We can observe that, the classification accuracy of the two models are comparable, but model \eqref{eq: classification_M2} obviously performs better
in terms of feature selections.

\begin{figure}[H]
	\vspace{-0.1cm}
	\flushright
	\includegraphics[width = 1\columnwidth]{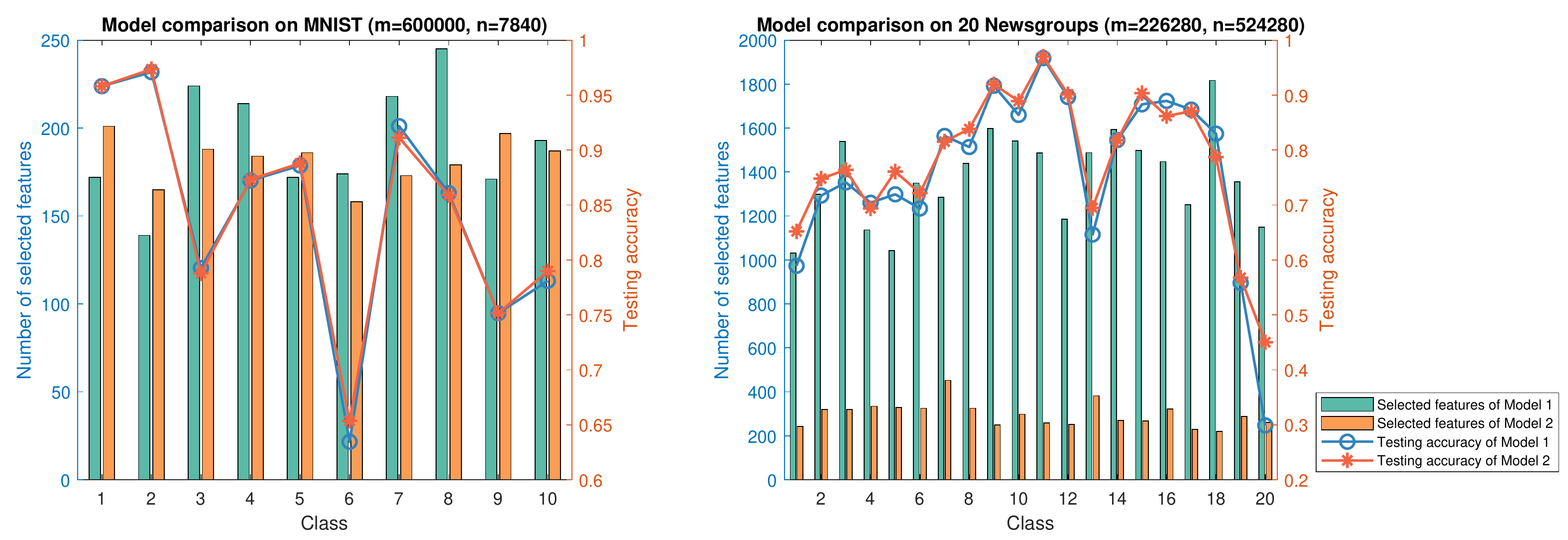}
	\setlength{\abovecaptionskip}{-15pt}
	\setlength{\belowcaptionskip}{-0pt}
	\caption{Model comparison between \eqref{eq: classification_M1} and \eqref{eq: classification_M2}. }
	\label{fig: real_model_compare}
\end{figure}

More importantly, as we can see in Table \ref{tab: numerical_results_realdata} and Figure \ref{fig: real_model_compare}, for the MNIST dataset, although the numbers of features in each class selected by two models are close, the total selected unique features of model \eqref{eq: classification_M2} is much less than that of model \eqref{eq: classification_M1}. This is because a group of important features which are selected by model \eqref{eq: classification_M2} are shared across different classes, which is consistent to our prior knowledge since almost all the targeted digits are located at the center of the images in the MNIST dataset. On the contrary, model \eqref{eq: classification_M1} selects 717 unique features out of the total 784 features, which means it selects a lot of uninformative features.

\begin{figure}[H]
	\begin{center}
		\includegraphics[width = 0.45\columnwidth]{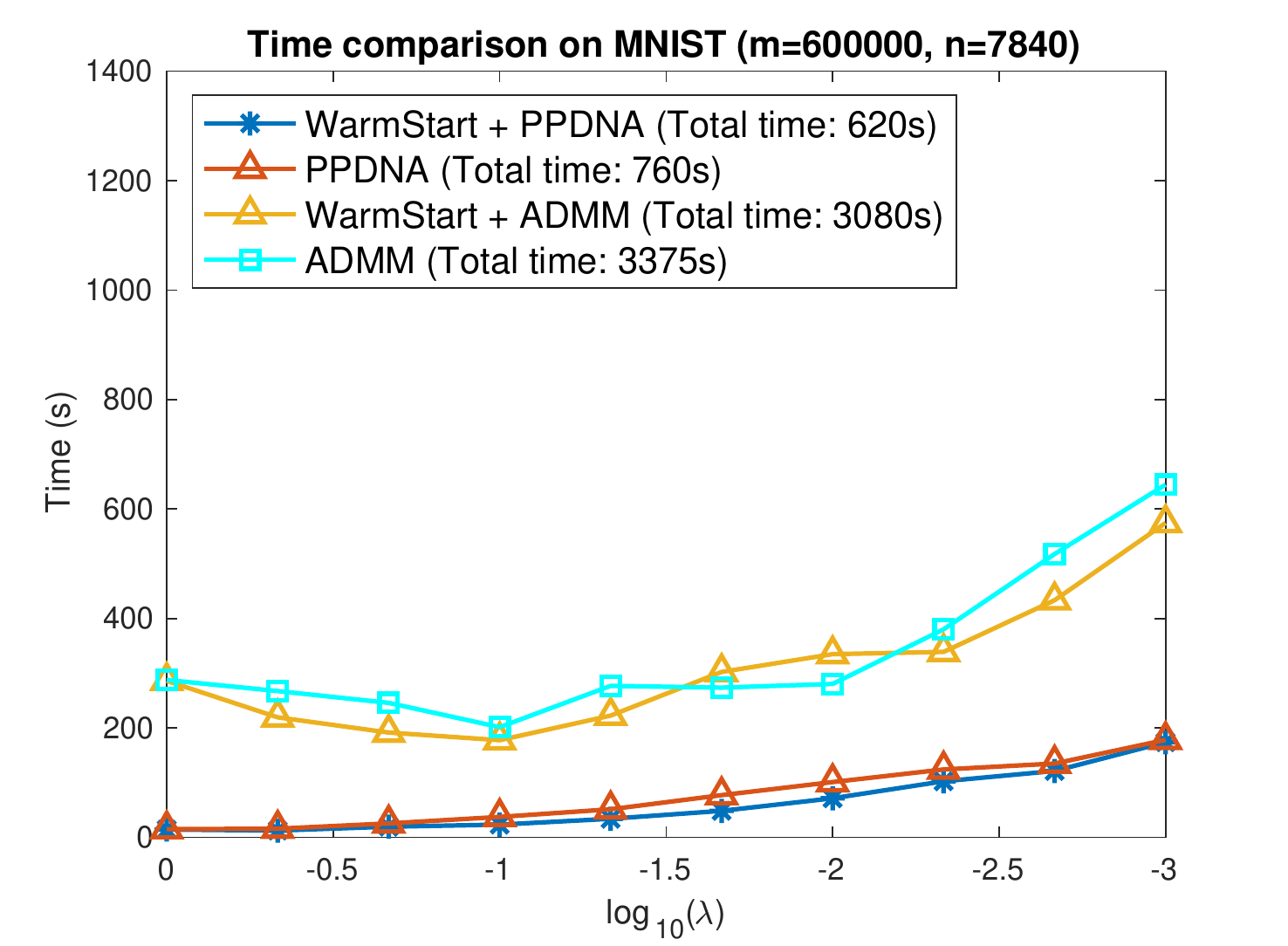}
		\quad
		\includegraphics[width = 0.45\columnwidth]{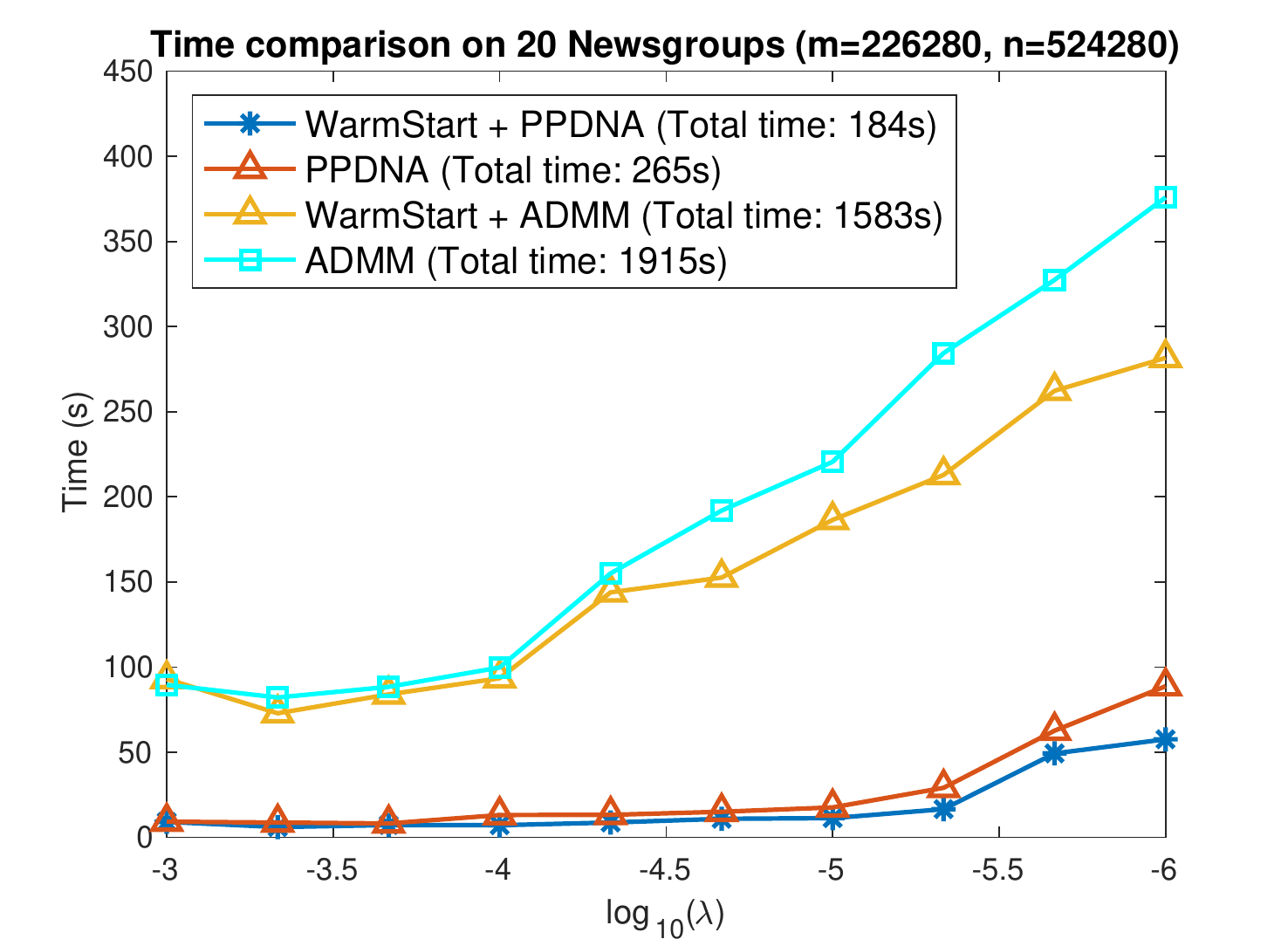}
		\setlength{\abovecaptionskip}{-2pt}
		\setlength{\belowcaptionskip}{-15pt}
		\caption{Time comparison on multi-class classifications.}
		\label{fig: real_data}
		\vspace{-0.6cm}
	\end{center}
\end{figure}

From now on, we focus on model \eqref{eq: classification_M2} and test the efficiency of our proposed PPDNA for solving the model with a sequence of hyper-parameters. For the two datasets, we generate solution paths for $\lambda$ over the range from 1 to $10^{-3}$ with 10 equally divided grid points on the $\log_{10}$ scale, and $10^{-3}$ to $10^{-6}$ with 10 equally divided grid points on the $\log_{10}$ scale, respectively. We compare the computation time for generating the solution path by PPDNA and ADMM with or without warm-start strategy. The results are summarized in Figure \ref{fig: real_data}. On the MNIST dataset, the warm-start strategy can reduce $18\%$ and $9\%$ of the computation time of PPDNA and ADMM, respectively. Moreover, PPDNA (with warm-start) is $4$ times faster than ADMM (with warm-start) on this dataset. On the 20 Newsgroups dataset, the warm-start strategy can accelerate PPDNA and ADMM by $31\%$ and $17\%$, respectively. Moreover, PPDNA (with warm-start) is around $8$ times faster than ADMM (with warm-start) on this dataset.

\section{Conclusion}
\label{sec:conclusion}

In this paper, we design a highly efficient and scalable dual Newton method based preconditioned proximal point algorithm to solve the exclusive lasso models, which is proved to enjoy a superlinear convergence rate. As important ingredients, we systematically study the proximal mapping of the weighted exclusive lasso regularizer and its generalized Jacobian. Numerical experiments show that the proposed algorithm outperforms the state-of-the-art algorithms by a large margin when solving large-scale exclusive lasso problems.


\section*{Funding}
Meixia Lin is supported by The Singapore University of Technology and Design under MOE Tier 1 Grant SKI 2021{\_}02{\_}08. Yancheng Yuan is supported by The Hong Kong Polytechnic University under Grant P0038284. Defeng Sun is supported in part by the Hong Kong Research Grant Council grant PolyU 153014/18P. Kim-Chuan Toh is supported by the Ministry of Education, Singapore, under its Academic Research Fund Tier 3 grant
call (MOE-2019-T3-1-010).



\end{document}